\documentclass{article}
\usepackage{tikz-cd}
\usepackage{mycommands}
\usepackage{lpic}
\usepackage{enumerate}
\usepackage{diagrams}
\usepackage{setspace}
\usepackage{url}
\usepackage[titletoc]{appendix}
\newcommand{\dl}[1]{}
\author{Dmitry Vaintrob}
\DeclareFontFamily{U}{BOONDOX-calo}{\skewchar\font=45 }
\DeclareFontShape{U}{BOONDOX-calo}{m}{n}{
  <-> s*[1.05] BOONDOX-r-calo}{}
\DeclareFontShape{U}{BOONDOX-calo}{b}{n}{
  <-> s*[1.05] BOONDOX-b-calo}{}
\DeclareMathAlphabet{\mathcalboondox}{U}{BOONDOX-calo}{m}{n}
\SetMathAlphabet{\mathcalboondox}{bold}{U}{BOONDOX-calo}{b}{n}
\DeclareMathAlphabet{\mathbcalboondox}{U}{BOONDOX-calo}{b}{n}
\usepackage{stackengine,xcolor}
\input pdf-trans
\newbox\qbox
\def\usecolor#1{\csname\string\color@#1\endcsname\space}
\newcommand\bordercolor[1]{\colsplit{1}{#1}}
\newcommand\fillcolor[1]{\colsplit{0}{#1}}
\newcommand\colsplit[2]{\colorlet{tmpcolor}{#2}\edef\tmp{\usecolor{tmpcolor}}%
  \def\tmpB{}\expandafter\colsplithelp\tmp\relax%
  \ifnum0=#1\relax\edef\fillcol{\tmpB}\else\edef\bordercol{\tmpC}\fi}
\def\colsplithelp#1#2 #3\relax{%
  \edef\tmpB{\tmpB#1#2 }%
  \ifnum `#1>`9\relax\def\tmpC{#3}\else\colsplithelp#3\relax\fi
}
\newcommand\outline[1]{\leavevmode%
  \def\maltext{#1}%
  \setbox\qbox=\hbox{\maltext}%
  \boxgs{Q q 2 Tr \thickness\space w \fillcol\space \bordercol\space}{}%
  \copy\qbox%
}

\bordercolor{black}
\fillcolor{white}
\def\thickness{.3}
\usepackage{mathrsfs}
\usepackage{amsbsy}
\renewcommand{\phi}{\varphi}

\newcommand{\spf}{\text{Spf}}

\newcommand{\ddo}{\mathring{\dd}}
\renewcommand{\vec}{\overrightarrow}
\renewcommand{\A}{\mathcalboondox{A}}

\renewcommand{\X}{\mathcal{X}}
\newcommand{\Y}{\mathcal{Y}}
\newcommand{\I}{\mathcal{I}}

\renewcommand{\O}{\mathcal{O}}

\renewcommand{\ll}{\mathfrak{M}}
\newcommand{\R}{\mathcal{R}}

\renewcommand{\angle}{\mathtt{angle}}
\newcommand{\comm}{\mathtt{Comm}}

\newcommand{\LogSch}{\mathtt{LogSch}}

\newcommand{\taut}{\mathtt{taut}}
\newcommand{\rotor}{\mathtt{rotor}}

\newcommand{\glue}{\mathtt{glue}}
\newcommand{\selfglue}{\mathtt{selfglue}}
\newcommand{\stabilize}{\mathtt{stabilize}}

\newcommand{\conf}{\mathrm{Conf}}

\newcommand{\reparam}{\mathtt{reparam}}
\newcommand{\fld}{\mathrm{FLD}}

\newcommand{\ffc}{{\pmb{\mathcal{M}}}}

\newcommand{\flc}{{\pmb{\mathfrak{m}}}}
\newcommand{\et}{{\text{\'et}}}

\newcommand{\ann}{\operatorname{Ann}}
\newcommand{\baran}{\overline{\operatorname{Ann}}}
\newcommand{\hourparam}{{\mathtt{hour}^\mathtt{param}}}
\newcommand{\forg}{\mathtt{forg}}
\newcommand{\hour}{\mathtt{hour}}
\newcommand{\Hour}{\mathcal{HG}}

\newcommand{\ggtm}{\mathbb{G}\mathrm{T}\mathcal{M}}
\newcommand{\ggam}{\mathbb{G}_\mathrm{A}\mathcal{M}}
\newcommand{\gglm}{\mathbb{G}_\mathcal{L}\mathcal{M}}

\renewcommand{\a}{\mathfrak{a}}

\newcommand{\Gr}{\operatorname{Gr}}

\title{Moduli of framed formal curves}

\begin{document}
\maketitle
\abstract{We introduce \emph{framed formal curves}, which are formal algebraic curves with boundary components parametrized by the punctured formal disk. We study the moduli space of nodal framed formal curves, which we endow with a logarithmic structure. We show that this moduli space is a smooth formal logarithmic stack. The remarkable property of our construction is that framed formal curves admit a natural operation of ``gluing along the boundary'' which works well in families and preserves smoothness (both in a formal and in a logarithmic sense), and this induces gluing maps on the level of moduli. 

Using moduli spaces of framed formal curves we enhance the operad $E_2$ of little disks (as well as its cousin, the \emph{framed} little disks operad) to a fully log motivic operad. We use this structure to obtain a purely algebro-geometric proof of the formality of chains on these classical operads (initially proven for little disks by Tamarkin using analytic methods). We also recover and extend the known Galois action on the $\ell$-adic cohomology of framed and unframed little disks, and on Drinfeld associators, and extend it to the action on integral chains of a larger group scheme: the logarithmic \emph{motivic Galois group}. Our methods generalize to a higher genus context, giving new ``motivic'' enrichments (for example, action by the absolute Galois group and by the log motivic Galois group) on the operad of chains in the \emph{oriented geometric bordism} operad of Ayala and Lurie, which encodes the algebraic structure on the Hochschild cochains of any fully dualizable DG category. 
}

\newpage
\tableofcontents
\newpage

\section{Introduction \dl{(a1)}}\label{sec:a1}\label{sec:intro}

The topological operad $LD$ of \emph{little disks}, introduced by May in \cite{may}, is responsible for the algebraic properties of double loop spaces $\Omega\Omega(X)$, in particular implying that their homology has \emph{Gerstenhaber algebra} structure (a homologically shifted Poisson structure). \emph{Deligne's conjecture} (now a theorem) states, essentially, that in a derived-algebraic context, the DG operad $C_*(LD)$ of \emph{chains} on $LD$ acts on the Hochschild chains $CH_*(A)$ of any associative algebra. This result is easy to prove using modern tools, in particular Ran spaces (see e.g. Lurie, \cite{lurie_ran}); see also \cite{getzler-jones}, corrected by \cite{mcclure-smith} and \cite{voronov}. However it was first proved in characteristic $0$ using a deep result of Tamarkin: that the operad of little disks is \emph{formal}, i.e.\ there is a chain of quasiisomorphisms of dg operads between chains $C_*(LD, \qq)$ and homology, $H_*(LD, \qq)$ with zero differentials. The motivation for Tamarkin's result was Kontsevich's study of \emph{deformation quantization}, the conjecture (long assumed by physicists before any mathematical proof) that the Poisson algebra of functions $C^\infty(M, \rr)$ on a Poisson manifold $M$ can be (formally\footnote{in an unfortunate coincidence, the adjective ``formal'' describing a nowhere convergent power series has nothing in common with the adjective ``formal'' describing a derived object equivalent to its homology.}) \emph{quantized}, i.e.\ extended to an associative law in formal power series over the parameter $\hbar$, with the Poisson structure on $C^\infty(M)$ describing its classical limit (i.e.\ its basechange to the first-order deformation algebra $\rr[[\hbar]]/\hbar^2 = 0$). Kontsevich showed in \cite{kontsevich_deformation} that deformation quantization follows from a certain strong form of the HKR theorem on Hochschild cohomology of affine space (essentially, the formality of the Lie-$\infty$ algebra structure on Hochschild homology), which he proved directly over $\rr$ using certain physically inspired integrals associated to graphs. 

It was observed by Tamarkin, \cite{tamarkin_quantization} (see also Hinich, \cite{hinich_tamarkin}), that Kontsevich's central argument would follow from formality of chains on the little disks operad, which he proved in \cite{tamarkin_formality} (a revised version of his \cite{tamarkin_squares}), using Drinfeld's theory of universal \emph{associators}, \cite{drinfeld}. Using similar techniques Tamarkin showed, in \cite{tamarkin_action} that, like Drinfeld's theory of associators, the operad of Gerstenhaber algebras (and hence, by formality, that of chains on little disks) has action by the \emph{Grothendieck-Teichm\"uller group}, $\hat{GT}$, a large (algebraic pro-unipotent) group whose points over any nonarchemedian field contains the absolute Galois group in a natural way. There are now multiple proofs of deformation-quantization, formality, and Grothendieck-Teichm\"uller action on little disks, and it is now known that explicit realizations of all of these are (in an appropriate sense) controlled by the same deformation-theoretic object, which can be interpreted on the one hand in terms of Drinfeld associators and on the other in terms of a certain graph complex (Willwacher, \cite{willwacher_action}). Using this relationship, Willwacher (\cite{willwacher_action}, section 10) finds a comparison between Kontsevich's and Tamarkin's universal quantizations. Similarly, it was observed by Petersen \cite{petersen} that by itself the action by the Grothendieck-Teichm\"uller group (or indeed, just by the Galois group $\Gamma_\qq\subset \hat{G}(\qq_\ell)$) implies formality, with splitting given by eigenspaces of certain Frobenius elements. Petersen's techniques can be extended to give partial formality results in positive characteristic: see \cite{horel_formality} (G. Horel and P. de Brito) and in a more general context \cite{cirici-horel} (J. Cirici and G. Horel). Despite the diversity of points of view, all proofs to date of geometrization-quantization and of the formality and Galois group action of chains on the little disks operad ultimately rely on one or another of two mysterious and transcendental analytic techniques: namely Kontsevich's graph integrals and Drinfeld's explicit construction of associators using the KZ equation. The present paper improves upon this status quo by introducing a new model for the operad of little disks as an algebro-geometric operad in a generalized sense and using motivic ideas. 

\subsection{Motives and a conjecture of Kontsevich}
Much less mysteriously, on the level of \emph{spaces} (ignoring operad structure), both the Galois and the Grothendieck-Teichm\"uller actions can be explained directly via algebraic geometry, using the philosophy of \emph{motives}. Namely, for any $n$ the space $LD_n$ of configurations of $n$ disks is canonically homotopic to the configuration spaces $\t{Conf}_n$ of $n$ distinct points in $\aa^1,$ which is an algebraic variety defined over $\qq.$

It has been known since Grothendieck and Artin's discovery of \'etale cohomology that if $X$ is a topological space homotopic to a smooth projective variety defined over $\qq$ then its cohomology has extremely rich additional structure. Its real cohomology $H^*(X, \rr)$ has a \emph{Hodge decomposition}, and its cohomology with $\ell$-adic coefficients $H^*(X,\qq_\ell)$ has an action of the Galois group $\Gamma_\qq :=  \t{Gal}(\qqbar/\qq),$ and the two structures are related in certain ways (for example, for almost all primes $p$ the eigenspaces in $H^i(X,\qq_\ell)$ of a $p$-Frobenius element in $\Gamma_\qq$ and the Hodge summands both split a certain filtration on $H^*(X, \qq)$). If $X$ is not smooth or projective but rather simply a variety of finite type over $\qq$ then, much (though not all) of the additional structure remains on the cohomology of the topological space $X(\cc)$. 

Grothendieck conjectured that all suitably functorial extra structure on $H^*(X)$ of a smooth projective variety is encoded by an object $[X]$ in an abelian category $\M$ of \emph{motives}, and that furthermore this category is accessible by (neutral) Tannakian formalism with Betti cohomology as the fiber functor: in other words, the category of motives is equivalent to the category of representations of a certain infinite-type group scheme $\gg_{\t{Groth}}$ called the \emph{Motivic Galois group}, such that $\M$ is equivalent to the category of finite-dimensional representations of $\gg_{\t{Groth}}$ with coefficients in $\qq$, and such that for any field $\Lambda$, the $\Lambda$-points $\gg_{\t{Groth}}(\Lambda)$ act on $H^*(X, \Lambda)$ in a functorial way compatible with the ring structure (equivalently, compatible with respect to product of schemes). In fact, if it exists $\gg_{\t{Groth}}$ can be fully characterized by the latter property --- so for example, the group $\gg_{\t{Groth}}(\qq_\ell)$ will contain the absolute Galois group, for any $\ell.$ 

Grothendieck's conjectures remain an active area of research, but in the \emph{derived} context, much more is known: there is a good notion of a DG category of (mixed) motives, equipped with a realization functor $X\mapsto [X].$ There is an elegant definition of a derived motivic Galois group scheme, $\ggam$, due to Ayoub \cite{ayoub1}, which acts on Betti cohomology \emph{on chain level} (i.e. for a ring $\Lambda,$ the points $\ggam(\Lambda)$ act in a DG sense on the Betti cochain ring $B^*(X, \Lambda)$ for any scheme $X$: see \cite{ayoub3}.) It is known that when the motive of a variety $X$ is of Tate type (for example, if $X$ has affine stratification), then the action of $\ggam$ factors through the smaller Tate motivic Galois group, which contains the Grothendieck-Teichm\"uller group $\hat{GT}$ as a subgroup.

Thus both the Grothendieck-Teichm\"uller equivariance and, as an easy consequence, the formality of the operad $LD,$ would be explained if the algebraic model $\t{Conf}_n$ for the spaces of operations $LD_n$ could be extended to the full operad structure. However there is a problem: the composition maps between the spaces $\t{Conf}_n$ cannot be made algebraic, even on the level of corresponces! For example an algebraic model for the operad composition map $LD_2\times LD_2\to LD_3$ would imply a nontrivial algebraic map $\gg_m\to \gg_m\setminus \{1\}.$ 

The question of how much of the motivic structure on the homology (and more generally, on chains) of $\t{Conf}$ extends to the operad $FLD$ remains unsolved. In \cite{kontsevich_operads} Kontsevich conjectured that chains on the operad of little disks, as well as Drinfeld associators, have action by the full Tate motivic Galois group. In an unpublished letter to Kontsevich \cite{beilinson_let} (that the author learned about in the process of writing this paper), Beilinson sketched a conjectural method for constructing the Hodge-theoretic filtrations associated to such a conjectural motivic action, observing that these would imply formality\footnote{Beilinson's letter used logarithmic \emph{analytic} techniques, which are related to the present paper's viewpoint.}. 

\subsection{Logarithmic schemes and motives}
One of the takeaways of the present work is that while the little disks operad does not have a nice algebro-geometric model in the category of schemes, it does have a model in a (derived enhancement of) Kato's category of \emph{logarithmic schemes} over $\qq$. The category of logarithmic schemes is an extension of the category of schemes given by essentially allowing, first, ``schemes with boundary'' which are pairs of the form $(X, D)$ for $X$ a scheme and $D\subset X$ a (sufficiently nice) Cartier divisor, and, second, basechanges of certain maps of such pairs. Most classical invariants of schemes extend naturally to logarithmic schemes (as do standard properties of maps, such as smooth, \'etale, etc.). In particular, the Betti chains functor $B_\zz:X\mapsto C_*(X(\cc), \zz)$ from schemes to complexes extends to a functor (which we denote the same way) $B_\Lambda:\X\mapsto C_*(\X^{an}, \Lambda),$ where for a log scheme $\X$ the topological space $\X^{an}$ is the \emph{Kato-Nakayama} (or Betti) \emph{analytification}.

The homology and cohomology of the analytification of a logarithmic variety over $\qq$ (known as Betti cohomology) is known to carry some of the same structure as the cohomology of a non-logarithmic scheme, both on the level of cohomology groups and chains. For example, Betti chains $C_*(\X^{an}, \qq_p)$ with coefficients in $\qq_p,$ have action by the absolute Galois group $\Gamma_\qq$ (\cite{kato-nakayama}, see Corollary \ref{thm:log_betti_comparison} for a chain-level proof). There is a de Rham cohomology with $\qq$ coefficients and a transcendental Betti-de Rham comparison (again, defined on chain level), implying a theory of periods. Indeed, it is conjectured (but currently not known) that the automorphism group of the Betti functor on log schemes is isomorphic to the automorphism group of the Betti functor on schemes. This would follow from an ongoing program by Vadim Vologodsky to show that that the DG category of log motives (suitably defined) is equivalent to the category of ordinary motives, see \cite{vologodsky}. We call this conjectural equivalence ``Vologodsky's conjecture''. Whether or not this is true, one can (formally) define an ($\infty$-categorical) log motivic Galois group stack $\gglm^D$ over the integers (see Appendix \ref{sec:mot-inv}), and our results imply that it acts on the operad of chains on the little disks operad (or more precisely, for $\Lambda$ a ring, the points $\gglm^D(\Lambda)$ act on chains with coefficients in $\Lambda$). If Vologodsky's conjecture is true, this implies action by the Ayoub motivic Galois group $\ggam$, which in this case factors through the Tate motivic Galois group, implies an extension of the $\hat{GT}$ action which is the most general extension one could hope for. Without assuming Vologodsky's conjecture, a consequence of the motivic action is that the absolute Galois group $\Gamma_\qq$ acts on chains on little disks with $p$-adic coefficients ($\zz_p$ or $\qq_p$). This is enough to prove formality, and to describe a new weight structure on the bordism operad.

\subsection{Log geometric operads}
The result about motivic Galois action on the operad of little disks falls out as part of an algebro-geometric story translating to the log algebro-geometric context two other well-known topological operads. The simplest case of this construction is not little disks themselves but rather their close cousin, the operad of \emph{framed} little disks, $FLD$, with spaces of operations consisting of configuration spaces of disks in a larger disk endowed ``framed'' by a choice of constant-speed boundary parametrization of each interior disk. Here the motivic action is a consequence of a direct identification, up to homotopy, of $FLD$ with (the underlying topological operad of) a log operad, which we call $FLD^{log}$. 

The operad $FLD$ of framed little disks is (equivalent to) the genus-zero piece of the larger operad $\mathrm{Bord}^{2,1}$ of \emph{oriented geometric bordisms} (see \cite{ayala} for a general homotopy-theoretic description of such geometric operads). To every point of the configuration $FLD_n$ there corresponds a closed curve with boundary (the complement in the complex unit disk to the interior of the embedded little disks), together with an (analytic) parametrization of the boundary, split into incoming (well-oriented) and outgoing (antioriented) boundary components. The operad composition operation can then be understood as an abstract geometric glueing procedure applied to an outgoing boundary component of one complex curve and the outgoing component of another. One then defines an operad (equivalent to) $\mathrm{Bord}^{2,1}$ by allowing the complex surface with boundary to have arbitrary genus, and allowing the boundary parametrizations to be arbitrary (oriented in such a way that precisely one is an ``outgoing'' boundary component). In fact, this operad structure in higher genus can be extended to an algebraic structure that carries more information, namely, a topological \emph{modular operad} (which subsumes both wheeled operad and PROP structures). The modular operad $\mathrm{Bord}^{2,1}$ appears in homological field theory and also in Lurie's cobordism hypothesis theorem \cite{lurie_cobordism} (at the two-dimensional level), where it essentially comprises the structure that acts on the Hochschild homology of a fully dualizable $\infty$-category (or more generally, fully dualizable object in an $(\infty,2)$-category.

Similarly our logarithmic operad $FLD^{log}$ is the genus zero part of a higher-genus operad $\mathrm{Bord}^{2,1,log},$ with analytification equivalent to $\mathrm{Bord}^{2,1}$. There is a minor difference: the operad $\mathrm{Bord}^{2,1}$ is an operad in \emph{logarithmic orbifolds}, i.e.\ logarithmic stacks which are locally finite quotients (generalizing geometric stacks of Deligne-mumford type). This causes mild technical hurdles about the appropriate notion of operad and modular operad in this context, which are treated in Appendix \ref{sec:app_mod_op}. Just as for log schemes, the logarithmic motivic Galois group $\gglm$ acts on Betti chains, implying action on the chain operad $C_*(\mathrm{Bord}^{2,1}, \zz)$. This leads in particular to action by the absolute Galois group on chains with $\ell$-adic coefficients, $C_*(\mathrm{Bord}^{2,1}, \qq_\ell)$, and implies an interesting new chain-level splitting of this operad by weights of Frobenius. 

\subsection{The motivic Tate curve}
Much like for the genus zero situation, it has long been known that the geometric bordism spaces classifying complex curves of arbitrary genus with parametrized boundary are ``motivic'', i.e.\ homotopy equivalent to algebraic varieties. But our result that the gluing operations on the level of moduli spaces are compatible with this motivic structure seems to be new in higher genus, with one exception. Namely, in \cite{hain}, Hain establishes that a certain pair of Lie algebras associated to noncommutative motives (understood as Lie algebras of a motivic \emph{fundamental} group with additional structure) over $\cc$ are related by a map compatible with mixed Hodge structure, geometrically described by a parametrization of the Tate curve (but not given by any map of actual algebraic varieties or stacks). In the course of the paper, he speculates whether this map extends to a map of motives over $\qq$ (see also \cite{luo_thesis} by Ma Luo, Hain's student). In the language of the present paper, we see that a map matching Hain's descrption of the Tate curve  (or rather a mild extension of it) on real blow-up spaces \emph{is} given by a map of \emph{logarithmic} stacks (and indeed, defined over $\qq$). To this map we can associate a map of (log) motivic Galois groups, containing much (though due to limitations of our knowledge of log motivic Galois groups, not all) of the structures studied by Hain, including compatibility of this map with periods. See Appendix \ref{sec:mot-inv} for a further discussion of this. Motivic comparisons of this nature have traditionally been difficult to prove (Hain's paper uses rather elaborate computations with certain explicit elliptic associators given by solutions to the universal KZB equation), and rich in applications: for example computations related to the motivic map in \cite{hain} are used by Francis Brown \cite{brown_zeta} to deduce interesting equalities of periods (between zeta values and certain modular integrals). It seems that the approach of the present paper is a good context for systematically producing interesting ``gluing'' comparisons of a similar sort between motives and periods associated to moduli spaces of all genera.

\subsection{Main geometric input: framed formal curves}
We construct the logarithmic operads $FLD^{log}$ and $\mathrm{Bord}^{log}$ in several steps starting with a certain formal-geometric moduli problem (over $\qq$), which poses independent interest, and whose construction takes up most of this paper. Namely, recall that a formal scheme is a topological space with a sheaf of profinitely topologized \emph{topological} rings, which is in a suitable sense an inductive limit of formal thickenings of a (finite-type) scheme. A formal curve is a formal scheme whose normalization is a union of ordinary (smooth) curves, possibly with punctures, together with copies of the \emph{formal disk}, $\dd : = \spf(k[[t]])$ (the notation $\spf$ indicating the formal spectrum of a topological ring). We say that a formal curve is \emph{nodal} if all of its singularities are (formally locally) nodes. We say that a formal nodal curve is \emph{stable} if each of the ``global'' components of its normalization is closed (no punctures) and has a discrete group of automorphisms which do not move preimages of singular points (we make no requirement on formal components). Now over a point, nodal stable formal curves are nothing new: they are classified by points of a certain Deligne-Mumford space (and similalry over any reduced base). However, they have a rich deformation theory. For example, the (nodal and stable) ``formal cross'' $\spf (k[[x,y]]/xy = 0)$ has an infinite-order deformation over $k[t]/t^N$ given by $\spf k[[x,y]]/xy = t,$ which is an unobstructed deformation but not a global one (it cannot be extended to $k[t]$ or an \'etale neighborhood). Thus the moduli problem of stable formal nodal curves is highly non-reduced. It is also not representable by any reasonable geometric object: the isotropy, containing the automorphism group of a formal disk, is simply much too large. Instead the objects we study are \emph{framed formal nodal curves}, defined to be (deformations of) stable nodal formal curves with certain additional stabilizing data, namely, that of a \emph{framing}. We define a framing of a formal curve $X$ to be a map from the ``punctured formal disk'', $\ddo : =\t{spec}_{top}k((u))$ to $X$ which (over geometric points) bijects to a localization every formal component (note that as $\ddo$ is not a formal scheme, care needs to be taken to make sense of this data). The data of framings gets rid of all isotropy coming from formal components, though it introduces infinite-dimensionality. Nevertheless, the moduli space of framed formal curves turns out to be an excellently behaved formal object: it is a smooth infinite-dimensional formal stack which is of Deligne-Mumford type, i.e.\ locally the quotient of a smooth formal scheme by a finite group.

Framings should be thought of as formal-geometric analogues of boundary parametrizations. In this sense, framed formal curves are a kind of dual construction to a classically studied construction which we call \emph{coframed} curves in this paper. Namely, we define a \emph{coframing} to be a parametrization, by the punctured disk $\ddo,$ of the punctured formal neighborhood of a smooth but non-proper (global) curve. Curves with such data are studied in e.g.\ \cite{benzvi_coframing}. Our construction however has an important advantage. Namely, given two framed formal curves $X, X'$ together with a choice of a framing on each $\phi, \phi'$, there is a \emph{glued} framed formal curve $\glue_{\phi\sim \phi'}(X, X')$ (something that is not possible in the context of co-framings)\footnote{Gluing of coframings fails since one cannot identify two formal disks with opposite orientations. Using framings, we get around this issue by first gluing in a global genus zero component between the two framings, then stabilizing it away.}. Moreover, our construction makes sense in families and induces an operation on moduli spaces of framed formal curves.

It is these gluing operations which are responsible for all the algebraic operations in the logarithmic operadic objects we define. To define these objects, we endow the moduli space of framed formal curves with a logarithmic structure relative to a certain normal-crossings divisor (classifying ``strictly'' nodal curves). The resulting logarithmic stack is called $\ffc_{g,n,f}$, and is a \emph{smooth} formal logarithmic stack of Deligne-Mumford type. Here $g$ is the genus, $f$ the number of formal (framed) components, and $n$ an additional parameter corresponding to a choice of $n$ marked points (this turns out to be a convenient additional handle on the geometry, just like it is for Deligne-Mumford moduli spaces). When no framings are chosen, the spaces $\ffc_{g,n,0}$ correspond to the standard logarithmic structures on the Deligne-Mumford moduli spaces $\bar{\M}_{g,n}$ relative to the normal-crossings boundary divisor.\footnote{The reason we do not also consider an ``open'' moduli space where the boundary divisor is ``removed'' is because when $f\ge 1$ the resulting object would no longer be an object of formal geometry, as every geometric point would be ``punctured''.} In many of our applications an even more useful object to consider is the object $\flc_{g,n,f}$ given by taking the \emph{reduced} log stack underlying $\ffc_{g,n,f},$ which is a log stack of finite type (though no longer smooth). The operads $FLD^{log}$ and $\mathrm{Bord}^{2,1}$ are built out of the objects $\flc_{g,n,f}$ with $n=0,$ together with their glueing maps.

\subsection{Statement of main results}
In terms of the derived log motivic Galois group $\gglm^D$ introduced above, we prove the following results. 
\subsubsection{The little disks operad}
\begin{thm}\label{thm:log_ld}
For $\Lambda$ a commutative (or connective $E_\infty$) ring, the DG operad of chains on little disks $C_*(LD, \Lambda)$ has (chain-level) action by the derived log motivic Galois group $\gglm^D(\Lambda).$ 
\end{thm}
\begin{cor}
The DG operad of chains with $\ell$-adic coefficients, $C_*(LD, \zz_\ell)$ has chain-level action by the absolute Galois group $\Gamma_\qq$.
\end{cor}
Formality with $\qq_\ell$ coefficients, and hence formality with coefficients in any characteristic-zero field follows from the corollary by diagonalizing action of a Frobenius element. 
\begin{cor}
Conditionally on Vologodsky's conjecture (\cite{vologodsky}), $C_*(LD,\qq)$ has action by the \emph{Mixed Tate} motivic Galois group, $\ggtm(\qq).$ 
\end{cor}
This last (conditional) corollary extends the known Grothendieck-Teichm\"uller action on $LD$ to a chain-level action of the full group of mixed Tate motives, implying the ``best possible'' motivic structure on the little disks operad. As structures that act naturally on little disks also act on Drinfeld associators, this completes (conditionally) the proof of the main conjecture in \cite{kontsevich_operads}. 

\subsubsection{The framed little disks operad}
\begin{thm}\label{thm:log_fld}
The operad $FLD$ of framed little disks is canonically equivalent (via a chain of homotopy equivalences of operads) to an operad in log schemes, $FLD^{log}$.
\end{thm}
\begin{cor} The motivic Galois group $\gglm^D(\zz)$ acts on chains $C_*(FLD,\zz)$.\end{cor}
\begin{cor}
The DG operad of chains with $\ell$-adic coefficients, $C_*(FLD, \zz_\ell)$ has chain-level action by the absolute Galois group $\Gamma_\qq$.
\end{cor}
As an immediate consequence we deduce the formality of this operad over any characteristic-zero field (first proven in \cite{gs_formality} and \cite{severa}), with splitting over $\qq_\ell$ given by diagonalizing a Frobenius in $\Gamma_\qq.$ We also note that conditionally on \cite{vologodsky}, the Motivic Galois group of mixed Tate motives $\ggtm(\qq)$ acts on $C_*(FLD,\qq)$. 


\subsubsection{The $2$-bordism operad}
\begin{thm}\label{thm:log_bord}
There exists an operad $\mathrm{Bord}^{2, 1, log}$ of logarithmic orbifols (stacks of Deligne-mumford type with log structure) whose analytification is $\mathrm{Bord}^{2,1}$.
\end{thm}
\begin{thm}\label{thm:log_bord_weight}
  The absolute Galois group $\Gamma_\qq$ acts canonically, in the derived category, on $C_*(\mathrm{Bord}^{2,1},\qq_\ell)$. This induces a weight filtration on $C_*(\mathrm{Bord}^{2,1}, \qq_\ell)$ whose associated graded is non-equivariantly isomorphic to $C_*(\mathrm{Bord}^{2,1}, \qq_\ell).$
\end{thm}
For the $g=0$ piece, the weight filtration is ``pure'', i.e.\ splits the homological filtration and implies formality. In higher genus, the filtration is ``mixed'' and implies a split filtration on $C_*(\mathrm{Bord}^{2,1}, \qq_\ell)$ which is independent of the homological filtration. In an upcoming paper we plan to describe this filtration more canonically, as the perverse filtration on a certain pushforward sheaf between formal varieties. 

\section{Acknowledgements}
This work grew out of joint work with Alexandru Oancea \cite{oancea_vaintrob} on operads and moduli spaces in an analytic context. Many of the ideas here presented were developed through our discussions and speculations, and the present paper owes its existence to him. A conversation with Akhil Matthew in Paris about operads and motives was also instrumental to this paper's conception. Other people whose discussions and suggestions have been valuable for this work are Vadim Vologodsky, Pierre Deligne, Roman Travkin (whose suggestion it was to use the Tannakian point of view on motives), Nick Rozenblyum, Pavel Etingof, Paul Seidel, Vladimir Hinich, Isabel Vogt, and Eric Larson. I'd like to express some of the inexpressible gratitude I have to my father, Arkady, for reading my drafts, and for sharing with me over the years his love and knowledge of many of the subjects that have, in some cases quite unexpectedly, found their way into this paper. Finally I'd like to thank Ilan Roth for generously providing an idyllic setting for writing this in his beautiful Berkeley home. 

\section{Notations and conventions \dl{(a2)}}\label{not-conv}\label{sec:a2}
\noindent {\bf Schemes} in this paper will be assumed separable and of finite type over a base field $k$ of characteristic $0$.

\noindent A {\bf stack} in this paper will denote any functor from finite-type schemes to groupoids satisfying \'etale descent\footnote{In fact all stacks we encounter also satisfy fppf descent.} (a more general notion than the one used in e.g.\ the stacks project \cite{stacks_project}, where stacks are required to have fppf covers by schemes). 

\noindent An {\bf orbifold}, also known as a {\bf Deligne-Mumford stack}, is a stack that has affine diagonal and an \'etale cover by schemes. 

An {\bf ind-orbifold} is a direct limit of orbifolds under locally closed immersions. The new moduli spaces we will construct will be ind-orbifolds, though in many cases this will be nontrivial to prove. 

\noindent {\bf Singular, nodal, unstable (curves).} All ``negative'' mathematical adjectives applied to curves such as ``singular'', ``nodal'', ``unstable'', etc., are understood to mean ``at worst singular'', ``at worst nodal'', ``at worst unstable'', and so on. If we need to specify that a curve is not smooth (or not stable, etc), we will say ``strictly singular'', ``strictly nodal'', ``strictly unstable'', etc.

\noindent {\bf Coordinates of linear schemes.} It will be convenient for us to work simultaneously with several distinct copies of $\aa^1$, $\gg_m,$ $\pp^1$ or the formal disk $\dd^1$ (whose definition is recalled in 3.1) and their products. When this happens we distinguish the copies by the corresponding coordinates. To this end we write $$\aa^1_t : = \spec k[t]$$ for the affine line with coordinate $t$, and similalry $$(\gg_m)_t : = \spec k[t, t^{-1}],$$ $$\pp^1_t = \aa^1_t\sqcup_{(\gg_m)_t}\aa^1_{t^{-1}},$$ $$\dd^1_t = (\aa^1_t)_0^\wedge.$$
When working over a base field $k$, we have multiplicative action of $k$ on the commutative groups schemes $\aa^1$ and $\dd^1$ and multiplicative action of $\zz$ on $\gg_m$. This gives meaning to the notations $$\aa^1\otimes_k V,$$ $$\dd^1\otimes_k V,$$ $$\gg_m\otimes_\zz A$$ for $V$ a $k$-vector space and $A$ an abelian group. We use the standard notation for projective spaces associated to a vector space: $$\pp(V) : = \frac{(\aa^1\otimes_k V)\setminus 0}{\gg_m}.$$

\noindent {\bf \'Etale local Zariski open.} Throughout the paper, we will be working with families $X$ over a base $S$. In all such cases, unless otherwise specified, properties and constructions having to do with $X$ will be understood \'etale locally over $S$. For example, a local \emph{Zariski open} $U\subset X$ will, in this context (and unless otherwise specified) mean an \'etale open in the total space which is a Zariski open over every geometric point of $S$.

\noindent {\bf Vector bundles and multi-line bundles.} A vector bundle $\nu/X$ is a locally free sheaf of $\oo_X$-bundles over $X$. Associated to a vector bundle is an affine space $\aa\nu: = \spec (\oo_X[\nu^*]/X)$ and a $GL_n$-torsor $\gg\nu.$ A \emph{multi-line} bundle $(\tau_1,\dots, \tau_k)$ is a collection of $k$ line bundles (equivalent to a $T$-equivariant bundle for the torus $T= \gg_m^k$). For $\boldsymbol{\tau} = (\tau_1,\dots, \tau_k)$ a multi-line bundle, define $\gg\boldsymbol{\tau} : = \gg\tau_1\times\dots\times \gg\tau_i$ (the torus torsor) and $\aa\boldsymbol{\tau}: = \aa\tau_1\times \dots \times \aa\tau_n$ (the affine bundle associated to their direct sum). 

\noindent {\bf Semigroups} are sets with multiplicative structure. {\bf Monoids} are unital semigroups. Semigroups, monoids, and groups in a category $\C$ with finite limits are objects $G$ with functorial (in the placeholder object $-$) semigroup, resp., monoid, resp., group structure on the set $\hom(-, G).$

\section{Families of nodal formal curves \dl{(a3)}}\label{a3}
\subsection{Formal nodal curves\dl{(a3.1)}}\label{a3.1}
We say that a topological (commutative) ring $R$ has \emph{adic} topology if there exists an ideal $I\subset R$ such that the topology on $R$ is the coarsest among those in which all elements of $I$ are topologically nilpotent. To an adic ring $R$ is associated an affine formal scheme $\spf(R),$ which is a ``formally locally ringed'' space, i.e. a topological space with a sheaf of adic topological rings which have a certain locality property. A formal scheme is a formally ringed space glued out of spaces of the form $\spf(R)$. The most important formal scheme is the \emph{formal disk}
$$\dd : = \spf(k[[x]]),$$
which can be thought of as the algebro-geometric analogue of the unit complex disk in complex geometry. The disk $\dd$ has one closed point, $0\in \dd,$ corresponding to the unique closed maximal ideal of $k[[x]]$.
\begin{defi}[\dl{a1}]\label{def:a1}
  A formal scheme $X$ is smooth of dimension $n$ if in a formal neighborhood of each $x\in X$ it is isomorphic to an $n$-dimensional formal disk, i.e. we have
  $$\hat{\oo}_x\cong k[[x_1,\dots, x_n]].$$
\end{defi}
We say that a formal variety $X$ is a singular formal curve if it is topologically noetherian and isomorphic to the formal neighborhood of a one-dimensional (possibly singular) curve in the formal neighborhood of any point $x\in X$. \emph{Smooth} formal curves (a subset of singular formal curves, by our convention in Section \ref{not-conv}) are not very interesting: namely, it is obvious that any geometrically connected smooth formal curve is either a smooth curve (no formal structure) or $\dd$ itself. Both notions extend evidently to families over a Noetherian base scheme $S.$ Things get a little more interesting once we weaken the condition of smoothness to one of nodality. Namely, define the formal cross, $$\dd^+ : = \spf\left(\frac{k[[x,y]]}{xy = 0}\right).$$
\begin{defi}[\dl{a2}]\label{def:a2}
  A formal curve $C$ is nodal if it is isomorphic to either $\dd$ or $\dd^+$ in the formal neighborhood of any point. 
\end{defi}
Classification of nodal curves turns out to reduce essentially to curves with marked points:
\begin{lm}[\dl{lm:a1}]\label{lm:a1}
  Over an algebraically closed field, any connected formal nodal curve $C$ is of one of the following three types:
  \begin{enumerate}
  \item $C = \dd$
  \item $C = \dd^+$
  \item $C = C_0\cup_{c_0, c_1, \dots, c_n} \dd\sqcup \dd\sqcup \dots\sqcup \dd,$ with $C_0$ a (possibly punctured) reduced connected nodal curve and $c_1,\dots, c_n$ distinct (unordered) smooth points on $C_0$. 
  \end{enumerate}
Here the number of copies of $\dd$ is $n$ and each is glued onto $C$ by identifying $0\in \dd$ with one of the marked points $c_i$. 
\end{lm}
The proof of the lemma is straightforward. Since formal neighborhoods are finer than \'etale neighborhoods, a formal variety over any field $K$ (of characteristic zero) is a formal nodal curve if and only if it is so after basechange to the algebraic closure $\bar{K}$. We will introduce new notation for the last class of formal curves:
\begin{defi}[\dl{a3}]\label{def:a3}
  For $C$ a (possibly punctured) reduced geometrically connected nodal curve with a divisor $c_1\sqcup c_2\sqcup \dots\sqcup c_n\in C$ consisting of smooth geometric points of multiplicity one (possibly permuted by the Galois group of $K$), write $$C^+_{c_0,\dots, c_n} : = C\cup_{c_0, c_1, \dots, c_n} \dd\sqcup \dd\sqcup \dots\sqcup \dd.$$
\end{defi}
The curves $\dd, \dd^+$, which are not of the form $C^+_{c_0,\dots, c_n}$, will be called \emph{strictly formal} nodal curves. We call an irreducible component of a formal curve $C: = C^+_{c_0,\dots, c_n}$ as above \emph{global} if it is an irreducible component of $C_0$, and \emph{formal} if it is one of the glued copies of $\dd$ or if $C$ is strictly formal. Given any formal curve $C$ over $k$, we define its reduced locus, $C_0$, to be the union of its global components, and define its strictly formal piece to be the union of all strictly formal components (in general, a geometrically disconnected formal curve). 

We say that a formal nodal curve $C$ over $k$ is \emph{proper} if it is either strictly formal or is isomorphic to $C^+_{c_0,\dots, c_n}$ for $C$ a \emph{proper} reduced nodal curve. The intuition to keep in mind is that the formal unit disk $\dd$ is similar to the ``closed'' unit disk; this will become clear later on as we consider its boundary. 

\subsection{Families of nodal formal curves \dl{(a3.2)}}\label{sec:a3.2}
Because of our classification lemma above, the classification of formal nodal curves gives us little more than a classification of classical nodal curves with marked points. However, their deformation theory is much more interesting. Let us set up the deformation problem. Say $S$ is a connected base scheme (not necessarily reduced), with reduced subscheme $S_0$. Recall that a formal affine scheme $U$ over $S$ is a flat sheaf of topological rings $\oo_U$ over $S$ with adic topology (locally) defined by some sheaf of ideals $I_U\subset \oo_U$. It is pro-proper if it is (after restriction to the formal neighborhood of any geometric point) a formal thickening of a proper scheme.
\begin{defi}[\dl{a4}]\label{def:a4}
  Let $S$ be a (possibly non-reduced) Noetherian scheme with reduced stratum $S_0\subset S$. A proper formal nodal curve $X/S$ is a (flat) family of pro-formal varieties $X/S$ such that for any $x\in X$ mapping to $s\in S$, we have either $\oo_x^\wedge \cong \oo_s^\wedge[[x]]$ (i.e. $X$ is smooth one-dimensional relative to $S$) or $\oo_x^\wedge \cong \oo_S[[x,y]]/xy = t$, for $t$ some nilpotent element of $\oo_s^\wedge.$ 
\end{defi}
The definition is equivalent to requiring that $X$ near $x$ is locally isomorphic to a formal neighborhood of the node of the reduced fiber in a (global) nodal curve over $S$. 
\begin{defi}[\dl{a5}]\label{def:a5}
We define a formal nodal curve $X/S$ to be the complement to a (relative) Cartier divisor $D\subset \bar{X}^s$ inside the smooth locus of a proper formal curve $\bar{X}$ with at worst nodal singularities. 
\end{defi}

Note that while we make the definition for a finite-type base $S$, it extends immediately to a base which is a formal scheme, by defining a formal nodal curve over a formal scheme $S$ to be a compatible collection of formal nodal curves over each finite type subscheme $S'\subset S$. 

The key observation is that the ``cross'' formal scheme $\dd^+$ (the nodal union of two disks) has the following (versal) deformation over the base $$\dd^{(e)}_\epsilon: = \spf (k[\epsilon]/\epsilon^{e+1}):$$
write $$\dd^+_\epsilon : = \frac{k[[x, y]]}{xy = \epsilon}.$$
These combine to a deformation over the formal disk $\dd_\epsilon$, with the remarkable property that the fiber over the generic point of $\dd_\epsilon$ is smooth. The geometric picture one should keep in mind for this generic fiber is as a closed complex annulus of modulus $-\log(\epsilon)$ (for $\epsilon$ some small complex number). 

\section{Framing, co-framing, and moduli \dl{(a4)}}\label{sec:a4}
\subsection{Embeddings of the punctured formal disk \dl{(a4.1)}}\label{sec:a4.1}
Let $k((t))$ be the field of Laurent series in $t$. We write informally $\ddo: = \spec_{top} k((t)),$ viewed as an object of the dual category of topological rings (note that as the topology on $k((t))$ is not adic, this is not a formal scheme). For $R$ an adic ring, define $$\hom(\ddo, \spf(R)): = \hom_{top}(R, k((t))),$$ a Hom between topological rings. This is a functor a priori defined on affine formal schemes, but from $k((t))$ being a field it follows that $\hom(\ddo, -)$ is a cosheaf in the Zariski topology, hence it makes sense to define $\hom(\ddo, X)$ for $X$ any Noetherian formal scheme. Similarly, for $S$ a (finite type) base scheme, we define $\ddo_S: = \ddo\times S,$ which we view as an object of the dual category of sheaves of topological quasicoherent rings over $S$, the ``spectrum object'' associated to the sheaf of topological coherent rings $k((t))_S$ over $S$. Given another sheaf of topological quasicoherent sheaf $\R/S$, the functor $U\mapsto \hom_U(\R_U, k((t))_S)$ defines a presheaf of sets over $S$ in the \'etale topology. Denote its sheafification by $\homu_S(\R, k((t)))$. This is an \'etale sheaf of sets. The functor $\spf(\R)\mapsto \homu_S(\R, k((t)))$ is (\'etale locally on $S$) a cosheaf in the Zariski topology on $\R$; hence, this functor on affine schemes over $S$ extends to a functor $$\homu_S(\ddo, -): \hat{\t{Sch}}_{Noeth}/S\to \t{Set}/S_{\et}$$ from Noetherian formal schemes to \'etale sets over $S$. 

Note that the fact that $k((t))$ is a topological field implies that the functor $R\mapsto \hom(R, k((t)))$ is a cosheaf not just for the Zariski topology on formal schemes, but more generally on topological rings. Specifically, we have the following evident lemma.
\begin{clm}
  Suppose that $R$ is a topological ring. Let $f, g\in R$ be two elements. Let $R_f, R_g, R_{fg}$ be the completed localizations of $R$ with respect to the three elements $f, g, fg$. 
Then the diagram of hom sets of topological rings $$\hom(R_{fg}, k((t)))\rightrightarrows \hom(R_f, k((t)))\sqcup \hom(R_g, k((t)))\to \hom(R, k((t)))$$ is an equalizer diagram. If $\R$ is a coherent sheaf of topological rings over a base $S$ with global sections $f,g$ the corresponding fibered diagram gives an equalizer diagram of \'etale sheaves over $S$. 
\end{clm}
\subsection{Framings and co-framings \dl{(a4.2)}}\label{sec:a4.2}
If $C$ is a one-dimensional formal variety, then a map $\ddo\to C$ can behave like an open embedding. In fact there are two different ways in which this can happen, both of which will be important to us. Suppose $S$ is an affine formal curve over $C$ which is flat with at most nodal singularities. 
\begin{defi}[\dl{a6}]\label{def:a6}
  A map $\phi:\ddo_S\to C_S$ is a \emph{framing} if the sheaf of rings on $C_S$ defined by $\phi$ is a localization of the sheaf of topological rings $\oo_C$.
\end{defi}
This can be checked on the level of geometric points, where the property of being a framing is equivalent to the map $\ddo\to C$ factoring through the natural map $\ddo\to \dd$, via a parametrisation $\dd\to X$ of a formal irreducible component of $C$. So a framing is locally modeled on the embedding $\ddo\to \dd,$ up to deformation.

A coframing is locally modeled on an extension of $\ddo$ in the other direction, namely, on the embedding of the punctured  disk as the neighborhood of a puncture of a \emph{global} component of a non-proper curve $X$.
\begin{defi}[\dl{a7}]\label{def:a7}
  A map $\phi:\ddo_S\to C_S$ is a \emph{coframing} if there is a separable partial compactification $\bar{C} : = C\sqcup x$, smooth near $x$, such that the map $\ddo\to C$ extends to an embedding $\dd\to \bar{C}$ sending $0\in \dd$ to $x.$ 
\end{defi}
This is equivalent to the map of rings $\oo(C)\to \oo(\ddo)$, on a suitable Zariski neighborhood, being the completion with respect to some adic topology.

It will be useful for us to have certain standard framing and coframing maps. 
\begin{defi}[\dl{a8}]\label{def:a8}
Let $\dd_t= \mathrm{Spf}k[[t]]$ be the formal disk. We write $\phi:\ddo\to \dd_t$ its \emph{standard framing} given on the level of functions by the function $$\phi^*:k[[t]]\to k((u)), \quad t\mapsto u.$$
\end{defi}
\begin{defi}\label{dist-coframings}[\dl{a9}]\label{def:a9}
Let $(\gg_m)_t$ be the one-dimensional torus. We define two coframings,  $\psi_0, \psi_\infty:\ddo\to \gg_m,$ dual to the homorphisms 
\begin{align}
\psi_0^*:k[t,t^{-1}]\to k((u)),\quad& t\mapsto u\quad&\t{and}\\
\psi_\infty^*:k[t,t^{-1}]\to k((u)), \quad& t\mapsto u^{-1}.&
\end{align}
These identify $\ddo$ with punctured neighborhoods of the boundary points $0,\infty\in \pp^1_t$ in the compactification, respectively.
\end{defi}

\section{Stable and unstable nodal framed formal curves \dl{(a5)}}\label{sec:a5}
Here we introduce (stable, marked) nodal framed formal curves, which will be our main geometric object of study. First we introduce the unstable version. Recall that as before, unstable means ``at worst unstable'', nodal means ``at worst nodal'', etc.
\begin{defi}[\dl{a10}]\label{def:a10}
Let $S$ be a base (of finite type over a characteristic-zero field $k$). Let $I$ be an \'etable index set over $S$. We say that a formal curve over $S$ with framing by $I$ is an unstable nodal \emph{framed formal curve} if it is ind-proper (i.e. all global components are closed), and if, over any geometric point $s$, every formal component is framed by a (necessarily unique) framing $\phi_i$ for some $i\in I_s$. Similarly, we say that a ($n$-times) marked framed formal curve over $S$ with framing by $I$ is a framed formal curve with $n$ distinct marked points disjoint from all nodes (including ones with formal components). 
\end{defi}
Usually, our indexing set will be a collection of distinct points, $\{1, \dots, f\}$ (for $f$ the number of framings). When this is the case, we will write framed formal curves as tuples $(X; \phi_1,\dots, \phi_f)$, where it is understood that $X$ is nodal, $\phi_i$ are framings, and both $X$ and the $\phi_i$ might be defined over a base $S$. It will be convenient for us to allow an additional piece of data, namely, a collection of smooth marked sections $x_1, \dots, x_n \in X$. Write $$\bar{\ffc}_{g,n,f}^{unst}(S)$$ for the groupoid of framed formal curves $(X; \phi_i; x_j)$ over $S$ with $f$ framings and $n$ smooth marked points. Then $\bar{\ffc}_{g,n,f}^{unst}$ is a functor from schemes to groupoids. 
\begin{lm}[\dl{lm:a2}]\label{lm:a2}
The functor $\bar{\ffc}_{g,n,I}^{unst}(S)$ is represented by a stack in the fppf topology. 
\end{lm}
By our convention about the meaning of ``stack'' in Section \ref{not-conv}, it suffices to check that the set of framed formal curves over $S$ satisfies \'etale descent. This is clear.

\subsection{Stability \dl{(a5.1)}}\label{a5.1} Let $(X; \phi_i; x_j)\in \bar{\ffc}_{g,n,f}$ be an (at worst) unstable framed formal curve over $S$. Here and later we will frequently suppress the $\phi_i, x_j$ from the notation and write ``$X$'' to mean the entire triple $(X; \phi_i; x_j)$. Now we say that $X$ is \emph{stable} if over every geometric point $s\in S$, every global irreducible component  of $X$ has at least one marked point or node if it has genus one and at least three if it has genus zero: equivalently, if the data $(X_s; \phi_i(s), x_j(s))$ has no infinitesimal automorphisms. Note that we need impose no conditions on formal components since they can have no automorphisms which preserve framing. As this condition is only checked on geometric points, it defines a substack of $\bar{\ffc}_{g,n,f}^{unst}$ (which we will see is an open substack).

\section{Gluing \dl{(a6)}}\label{sec:a6}
\subsection{Asymmetric gluing \dl{(a6.1)}}\label{sec:a6.1}

Say $S_0$ is a \emph{reduced} base. Say $X, \phi$ is a formal curve over $S_0$ with a single framing and $Y, \psi$ a formal curve with a single coframing. To the coframing $\psi$ there corresponds a partial compactification $Y^+$ which is the unique curve formed by adding a single smooth point to $Y$ in such a way that $\psi$ extends to a map from the (unpunctured) formal disk. On the other hand, in the case of the framed curve $X$ over a reduced base, $\phi$ already extends to a map $\bar{\phi}:\dd\to X$ which is an isomorphism with a proper component of $X$. Let $X^-$ be the curve obtained by removing the formal component but leaving the attaching point (equivalently, replacing $X$ by its reduced locus in a sufficiently small Zariski neighborhood of $\phi$). Let $x_0, y_0$ be the points of $X^+, Y^-$ at which the structure was changed. In this special context, we define the asymmetric gluing $$X\overset{\circ}{\cup}_{\phi\sim \psi} Y$$ to be the nodal curve $X^+\cup_{x_0\sim y_0} Y^-.$ Now over a general base $S$, we define the asymmetric gluing $X\cup_{\phi\sim \psi}Y$ to be a thickening of $X_{S^0}\overset{\circ}{\cup}_{\psi\sim\phi} Y_{S_0}$ (gluing of restrictions over $S^0$), defined locally (in the Zariski topology of $X_{S^0}, Y_{S^0}$, which, recall, consists of \'etale opens which are Zariski over every geometric point of $S$) by the following formula.
\begin{defi}[\dl{a11}]\label{def:a11}
  Let $(X, \phi)$ be an \emph{affine} formal curve with a framing and $(Y, \psi)$ an \emph{affine} nodal formal curve with a coframing, all flat over an affine base $S = \spec(R)$. Then we define the \emph{asymmetric} gluing $X\cup_{\phi\sim \psi} Y$ to be the variety $\spf(\oo_X\times_{R((x))} \oo_Y)$.
\end{defi}
  This definition is evidently local: i.e., replacing $X, Y$ by \'etale neighborhoods through which the framing, respectively, coframing, factors will produce an \'etale neighborhood of the gluing. Thus it patches well on the Zariski topology on $X_{S^0}\cup_{\phi\sim\psi} Y_{S^0}.$

Geometrically and if $S$ is a point, this corresponds to replacing the framing disk $\bar{\phi}$ by the partial closure of $Y$ in a neighborhood of $\psi$, thus creating a node; over general $S$, the framing-coframing pair precisely gives enough data to determine a deformation of this nodal curve. This curve inherits a map from $\ddo,$ but this is now neither a framing nor a coframing, but rather lands in the interior of the global component of the glued curve.

By locality we can extend this definition to the case where $X, Y, S$ are not necessarily affine. Similarly, if $X$ is a variety with framing and coframing, $(X, \phi, \psi)$, then we have a notion of self-gluing $X_{\phi\sim \phi}$ given \'etale locally by $X\cup_{\phi\sim \psi} X$ and then identifying the two copies of $X$ at the complement of the gluing point. More generally, given any possibly disconnected nodal curve $X$ over a base $S$, together with a framing-coframing pair given by the \'etale sets $(\t{\'I}, \t{\'O})$, together with an isomorphism of \'etale sets over $S$, which we write $\t{\'I}\sim \t{\'O}$, we can define (by applying the above construction \'etale locally) a new self-glued formal curve $X_{\t{\'I}\sim \t{\'O}}$. Any disjoint framings and coframings will be preserved by this procedure. In this paper, we will only consider discrete \'etale sets over a base, namely $\t{\'I}=\t{\'O} = \{1,\dots, n\}.$

\subsection{Symmetric gluing \dl{(a6.2)}}\label{sec:a6.2}
In the sequel, we will mostly consider \emph{symmetric} gluing, where both the input and the output curves are framed rather than coframed. This will be accomplished by interpolating between coframings using the distinguished coframings $\psi_0, \psi_\infty$ on $\gg_m$ introduced in Definition \ref{dist-coframings}. 

\begin{defi}[\dl{a12}]\label{def:a12}
Given a pair of curves $X, X'$ both with framings $\phi, \phi'$, we define the ``symmetric gluing'' $$\glue_{\phi\sim \phi'}(X, X') : = \big(X\sqcup \gg_m \sqcup X'\big)_{\phi\sim \psi_0, \phi'\sim \psi_\infty}.$$
\end{defi}

If $S$ is a point and $X_0, X'_0$ are the reduced loci, obtained from $X, X'$, respectively by removing the formal components corresponding to $\phi, \phi'$, respectively (but keeping the attaching points), and $x\in X_0, x'\in X'_0$ are the attaching points, then $\glue_{\phi\sim \phi'}(X, X')$ looks like the nodal curve obtained by gluing a copy of $\pp^1$ in between $X_0, X'_0$, with $0\in \pp^1$ identified with $x$ and $\infty\in \pp^1$ identified with $x'$. Over a non-reduced scheme, the result of symmetric gluing will be a deformation of this nodal curve. Note that once we introduce a stability condition into this picture in the next section, the intermediate $\pp^1$ component will get ``blown down'' and the gluing identified with $X_0\sqcup_{x\sim x'} X'_0$ on the reduced locus. 

\section{Moduli \dl{(a7)}}\label{sec:a7}
The moduli stack $\bar{\ffc}_{g,n,f}^{unst}$ and its stable substack $\bar{\ffc}_{g,n,f}$, which will be studied in the next section, is much larger than the corresponding Deligne-Mumford moduli stack, but surprisingly is not more singular. Indeed, it can be locally (in an \'etale sense) modeled on a Deligne-Mumford stack multiplied by an infinite-dimensional formal disk, a statement which will be made precise in Theorem \ref{thm:torsor-str}. We will prove this remarkable fact by using asymmetric gluing to turn a formal curve into a global curve with some additional structure. In the next two subsections we give a taste, in genus zero, of what kind of formal parameters the additional structure introduces (essentially, it consists of formal deformations of an embedding of an open genus zero curve into a nodal one). In Section \ref{sec:hourglass-gen}, we write down a new model for framed formal curves of all genera in terms of something we call \emph{hourglass spaces}, $\Hour_{g,n,f}$. The formalism of hourglass spaces will also be useful in the following sections.

\subsection{Automorphisms of $\aa^1$ which fix two points \dl{(a7.1)}}\label{sec:a7.1}
Given any curve $X$ over an algebraically closed base  $k$ with marked points $x_1, \dots, x_n$ we define the \emph{automorphism stack} $\aut(X; x_i)$ of $(X; x_1,\dots, x_k)$ to be the stack (in our sense, see Section \ref{not-conv}) whose value on a base $S$ classifies fiberwise automorphisms of the constant family $S\times X$ over $S$ which fix all $x_i.$ Note that since this classification problem has no automorphisms, this is a set-valued (rather than a groupoid-valued) functor, and $\aut(X; x_i)$ is a stack only insofar as it might not be representable by a finite-type scheme. And since its value on any scheme is a group (in a functorial sense), it is a group object. Indeed, if $X$ is a proper curve, $\aut(X; x_i)$ is a reduced algebraic group of finite type (in fact a finite group if the marked curve $(X; x_i)$ is stable). But if $X$ has punctures, then this group is non-reduced, and while its reduced component has finite type there are infinitely many nilpotent directions. We will need to work with only two such groups, in a sense the simplest nontrivial ones. In the next section we will consider automorphisms of $\gg_m$ (without  marked points), and in this section we conside the group $$\A_{0,1} : = \aut(\aa^1; 0, 1).$$ Evidently, $\A_{0,1}$ has one reduced point. It will be convenient to understand coordinates on this group. Namely, let $M_{\aa^1}$ be the monoid of all polynomials $f\in k[x]$ under composition of polynomials (viewed as an ind-scheme, a limit of finite-dimensional affine schemes corresponding to finite-dimensional subspaces of $k[x]$). Let $M_{0,1}$ be the sub-monoid of polynomials which fix $0, 1\in \aa^1,$ which can be written $\{x + x(x-1)f(x)\mid f(x)\in k[x]\},$ once again a monoid under composition. This is a smooth infinite-dimensional affine variety. We have the following straightforward result.
\begin{lm}[\dl{a3}]\label{lm:a3}
The group scheme $\A_{0,1}$ is the formal neighborhood of the identity function $x$ in $M_{0,1}.$
\end{lm}
Note that $\A_{0,1}$ is a group ind-scheme with a single point, which means that it is a \emph{formal group} and, over a base of characteristic zero, it is uniquely determined (via the formal exponential map) by its Lie algebra. This Lie algebra is canonically identified with $T_{0,1}(\aa^1),$ the Lie subalgebra of the Lie algebra of vector fields on $\aa^1$ which vanish on the divisor $0\cup 1$. Though we will not need it, this point of view generalizes to more general open marked curves $(X; x_i)$ which have no automorphisms over a point.

\subsection{Framed Formal Nodal Annuli \dl{(a7.2)}}\label{sec:a7.2}\label{sec:annuli}
Define a (framed, formal) nodal annulus over $S$ to be a framed formal curve $(X; \phi_1,\phi_2)$ classified by a map $S\to  \bar{\ffc}_{0,0,2}$, i.e.\ a stable genus zero curve with no marked points and two boundary components. Note that stability is equivalent to $X$ having no global components, i.e. being a deformation of the formal cross $\dd_+$. The moduli space of annuli is a good testing ground for understanding classification and gluing of framed formal curves: in a sense, all of the geometry of the spaces $\bar{\ffc}_{g,n,f}$ is locally an amalgamation of the geometry of Deligne-Mumford spaces and the geometry of our moduli space of annuli (which will be endowed with the structure of a semigroup under stable gluing in the next section). 

Recall the notation $\A_{0,1} : = \aut(\aa^1, 0, 1).$
\begin{thm}[\dl{a4}]\label{thm:a4}
The space of (framed, formal) nodal annuli is isomorphic to the product $\dd\times \A_{0,1}^2.$ Further, it is possible to choose an identification in such a way that strictly formal nodal annuli lie over $\{0\}\times \A_{0,1}^2\subset \dd\times \A_{0,1}^2.$
\end{thm}
\begin{proof}
We prove this using the asymmetric gluing operation. Namely, let $\aa^1_t = k[t]$ be the affine line co-framed  by the map $\psi:\ddo\to \aa^1$ with $\psi^*: t\mapsto u^{-1}$ (coframing a neighborhood of $\infty$). Similarly, let $\aa^1_{t'} = k[t']$ be another copy of $\aa^1,$ again endowed with the same coframing $\psi'$ with $(\psi')^*:t'\mapsto u^{-1}$. Now for $(A, \phi_1, \phi_2)$ a framed nodal annulus over some base $S$, let $\hour(A, \phi_1, \phi_2)$ be the glued (unstable) closed variety of genus zero defined as the gluing $\hour(A, \phi_1, \phi_2) : = \aa^1\cup_{\psi_1\sim \phi_1} A \cup _{\phi_2\sim \psi_2}\aa^1$, so called because it looks like an ``hourglass'' deforming the ``pinched'' hourglass $\hour(\A_0) = \pp^1\cup_{\infty\sim 0} \pp^1$ with neck deformed to $A$. 

Now we endow the unstable genus zero curve $\hour(A)$ with some additional structure. Namely, first, we mark the points $(t=0, t=1, t'=0, t'=1)$ (which we abbreviate as $\{0, 1, 0', 1'\}\subset \hour(A)$) to produce a (now stable!) genus zero curve with four marked points. Second, we keep track of the two maps $\alpha:\aa^1\to \hour(A);\quad \alpha':\aa^1\to \hour(A)$ given by the glued copies of $\aa^1$. This gives a map $\hourparam$ from the moduli space $\baran$ to the stack classifying triples (over some base $S$): $$\left((X; 0,1,0', 1')\in \bar{\M}_{0, 4};\quad \alpha, \alpha': \aa^1:\to X\right),$$ with the following conditions:
\begin{enumerate}
\item $\alpha:\aa^1\to X$ and $\alpha':\aa^1\to X$ are nonintersecting open embeddings (of course, this can only happen when X is a deformation of the nodal genus zero curve with four marked points), 
\item $\alpha(0) = 0, \alpha(1) = 1, \alpha'(0) = 0', \alpha'(1) = 1'.$
\end{enumerate}
\begin{defi}[\dl{a13}]\label{defi:a13} Write $\Hour_{0,0,2}$ for the stack classifying the data above\footnote{Note that the marked points $0, 1, 0', 1'$ are a red herring: they are uniquely determined by the maps $\alpha, \alpha'$ together with the conditions $\alpha(0) = 0,$ etc. The reason for carrying this redundant structure around is to ensure stability of the underlying marked curve which will make our life easier by preventing the neccessity of working with large stabilizers.}. \end{defi}
The fact that this moduli functor satisfies fppf descent is obvious. We will soon see that this stack is an ind-scheme. 

Note that the map $\hourparam:\bar{\ffc}_{0,0,2}\to \Hour_{0,0,2}$ is a bijection: namely, if $(X; \{0,1,0', 1'\}; \alpha, \alpha')$ is a triple over $S$, then over the reduced locus of $S$ the space $X$ is a genus zero curve with a single node; call this point $\infty_0\subset X$ (a section of the restriction of $X$ to $S_0$). Now the formal neighborhood $\hat{\infty}$ of $\infty_0$ in $X$ is an unframed annulus over $S$ (it now has sections over all of $S$,  not just $S_0$). Now the maps $\psi, \psi': \ddo\to X$ factor (uniquely) through $\hat{\infty}$, hence the triple $(\hat{\infty}, \psi\circ\alpha, \psi'\circ\alpha')$ is a framed nodal annulus! It is easy to see that this construction is inverse to $\hourparam:\bar{\ffc}_{0,0,2}(S)\to \Hour_{0,0,2}(S)$, and so the stacks $\bar{\ffc}_{0,0,2}$ and $\Hour_{0,0,2}$ represent the same functor (in $S$) and are canonically isomorphic.

Now note that the group $\A_{0,1}^2$ acts on $\Hour_{0,0,2}$ in an obvious way: namely, given any section $(\gamma_1, \gamma_2): S\to \A_{0,1}^2$ and an element $(X; \alpha, \alpha'; 0, 1, 0', 1')\in \Hour_{0,0,2},$ we can precompose $\alpha, \alpha'$ with $\gamma', \gamma'$ respectively twisting $\gamma, \gamma'$ but leaving invariant the marked curve $(X; 0,1,0',1')\in \bar{\M}_{0,2}(S)$. Now it is obvious that $\A_{0,1}^2$ acts without stabilizers and that the quotient $\Hour_{0,0,2}\to\Hour_{0,0,2}/\A_{0,1}^2$ factors through the forgetful map $$\forg:\Hour_{0,0,2}\to \bar{\M}_{0,4}\cong \pp^1.$$ Now on any reduced scheme $\Hour_{0,0,2}$ has a single value, lying over the boundary point $\delta_0 \in \bar{\M}_{0,4}$ corresponding to the unique stable nodal curve, $(X_0; 0,1,0',1')$, of genus zero with four marked points such that $0, 1$ are in one irreducible component and $0', 1'$ are in the other. Thus the map $\forg$ has as image the formal neighborhood $\hat{\delta}$ of $\delta_0\in \pp^1$, which is isomorphic to the formal disk $\dd$. To conclude the proof of the theorem it remains to construct a section $\hat{\delta}\to \Hour_{0,0,2}.$ This we do explicitly. First, we introduce a (standard) local coordinate near $\delta$ in $\bar{\M}_{0,4}$, defined as follows: define the curve $X_t$ to be the compactification in $\pp^1\times \pp^1$ of the curve in $\aa^1\times \aa^1$  defined by $xy = t.$  Let $0,1$ be the points $(\infty,0), (1, t)$ respectively and let $0', 1'$ be the points $(0, \infty), (t, 1),$ respectively. That this is a local coordinte can be checked for example by identifying (for nonzero $t\in \aa^1$) $X_t$ with $\pp^1$ and computing the cross-products (one of which will be $t$). Now for each such curve define $H_t\in \Hour_{0,0,2}$ to be the tuple $(X_t; \alpha,\alpha'; 0,1, 0', 1')$. As we have defined $(X_t; 0,1,0',1'),$ it remains to define $\alpha, \alpha'.$ Let $\pi:X_t\to \pp^1$ be the projection to the $x$ coordinate (using $X_t\subset \pp^1\times \pp^1$) and $\pi':X_t\to \pp^1$ the  projection to the $y$ coordinate. Then in a formal neighborhood of $t=0,$ the map $\pi$ is an isomorphisms with $\aa^1$ when restricted to  the preimage $X_t\cap \aa^1_{x^{-1}}\times \pp^1_y$ of the complement $\aa^1_{x^{-1}} \cong \pp^1_x\setminus 0$ and, similarly,  $\pi'$ is an isomorphism when restricted to $X_t\cap \pp^1_x\times \aa^1_{y^{-1}}.$ The inverses to these restricted projections provide us with maps $\alpha, \alpha',$ concluding the proof of Theorem \ref{thm:a4}. 
\end{proof}
In particular, $\bar{\ffc}_{0,0,2}$ is a smooth infinite-dimensional formal manifold (in the sense of locally being isomorphic to an infinite power of $\dd$), something that will carry over (in an \'etale sense) to all $\bar{\ffc}_{g,n,f}$. The strictly nodal piece $\bar{\ffc}_{0,0,2}$ is in this case a codimension-one smooth submanifold; in general, it will be normal-crossings.

Note that each of the two actions of the group $\A_{0,1}$ is a piece of the action by the much larger group, namely the group $\aut(\ddo)$ of all automorphisms of the formal disk, acting by modifying a framing (a priori, this group has some additional profinite ``topological'' structure not captured by the classification problem: however, for a framed curve over any base this structure is cancelled out after quotienting out the stabilizer of the action). Using a slightly different subgroup gives a simpler description in coordinates of the universal family of annuli.
\begin{cor}[\dl{cor:a5}]\label{cor:a5}
Let $X_t;\phi_t, \phi_t'$ be the family over the formal disk $\dd_t$ given by $$X_t = \spf (k[[x,y]]/xy = t);\quad \phi^*(x,y) = (u, tu^{-1}), \quad (\phi')^*(x,y) = (tu^{-1}, u).$$
Let $\A_{neg}\subset \aut(\ddo)$ be the subset (not a subgroup!) consisting of automorphisms of the form $f(u) = u + e(u^{-1}),$ for $e$ a polynomial (by nilpotence constraints, over a base $S$, the polynomial $e$ must have coefficients in the nilpotent ideal  $I^{nilp}$). Then the family of framed formal nodal annuli over $\A_{neg}\times \dd_t\times \A_{neg}$ with fiber over $(\alpha, t, \alpha')$ given by $(X_t; \phi\circ\alpha; \phi'\circ \alpha')$ is a universal family of annuli: i.e. it is classified by an isomorphism $$\A_{neg}\times \dd_t\times \A_{neg}\cong \bar{\ffc}_{0,0,2}.$$
\end{cor}
\begin{proof}
The family is classified by some map $\iota:\A_{neg}^2\times \dd\to \bar{F}_{0,0,2}.$ As both sides are infinite-dimensional formal disks, it is sufficient (by Hensel's lemma, which still applies in this case) that $\iota$ is an isomorphism on tangent spaces. Note that further, $\iota$ maps $\A_{neg}^2 \times \{0\}_t$ to the locus of strictly nodal annuli, $\bar{F}^\nu_{0,0,2}.$ Since $\iota(\dd_t)$ is evidently transversal to the locus of strictly nodal annuli, it suffices to show that modifying the standard framed nodal annulus $$(\dd_+;\phi,\phi') : = \spf(k[[x,y]]/xy=0);\quad \phi^*(x,y) = (u,0), (\phi')^*(x,y) = (0,u)$$ by reparametrizations in $\A_{neg}^2$ classifies all annuli over $k[t]/t^2 = 0.$ Now evidently (locally), all nodal annuli are obtained by applying a reparametrization in $\aut(\ddo)^2$ to the standard framing on $\dd_+$. Such a reparametrization is isomorphic to the initial annulus iff it  lands in $\aut(\dd)^2\subset \aut(\ddo)^2,$ and the corollary follows from the observation that $\A_{neg}\subset \aut(\ddo)$ is a set of coset representatives for $\aut(\dd)\subset \aut(\ddo).$ Alternatively, one could prove this by seeing that $\A_{01}\mid \ddo$ (the restriction of automorphisms of $\aa^1$ to the neighborhood of infinity) is an orthogonal subgroup in $\aut(\ddo)$ to $\aut(\dd)$; equivalently, that locally, for $f\in \A_{0,1}(S)$ running over the set of automorphism of $\aa^1$, the principal part of $u^{-1} f(u^{-1})^{-1}$ can be $u^{-1}$ times any polynomial in $u^{-1}$, a straightforward computation. 
\end{proof}

\subsection{The hourglass space construction in higher genus \dl{(a7.3)}}\label{sec:a7.3}\label{sec:hourglass-gen}
A straightforward generalization of the hourglass construction $\hourparam$ gives us a powerful tool for understanding the local geometry of any moduli space of (nodal) framed formal curves, whether stable or unstable, captured by the following theorem. {\bf Suppose $(g,n,f)\neq (0,0,1)$.}
\begin{thm}[\dl{thm:a6}]\label{thm:torsor-str}\label{thm:a6}
Let $\iota:\bar{\M}_{g, n+f}^{unst}\to \bar{\M}_{g, n+2f}^{unst}$ be the map of moduli spaces of unstable (closed) nodal curves taking a curve $X; x_1,\dots, x_n, y_1, \dots, y_f$ by gluing a triply marked $\pp^1$ to each of the $y_j.$ This is a closed embedding of stacks; let $(\bar{\M}_{g, n+2f}^{unst})^\wedge$ be the formal neighborhood of this embedding. 
Then $\bar{\ffc}_{g,n,f}^{unst}$ is canonically a torsor over $(\bar{\M}_{g, n+2f}^{unst})^\wedge$ under action by the $f$-fold power $\A_{0,1}^f$ of the group $\A_{0,1} = \aut(\aa^1, 0,1)$; in particular, there is a canonical map $$\hour:\bar{\ffc}_{g, n, f}^{unst} \to (\bar{\M}_{g, n+2f}^{unst})^\wedge$$ with fibers isomorphic to the infinite-dimensional formal disk $\A_{0,1}^f$. Further, this map takes (the open substack of)  stable framed curves to stable marked curves and (the closed substack of) strictly unstable framed curves to strictly unstable marked curves.
\end{thm}
\begin{proof}
Let $(X; x_i; \phi_i)$ be an unstable framed formal curve in $\bar{\ffc}_{g,n,f}$. Define $$\hour(X)\in \bar{\M}_{g, n+2f}^{unst}$$ to be the curve in $\bar{\M}_{g,n + 2f}$ given by gluing a copy of $(\aa^1; 0,1)$ asymmetrically, via the standard coframing of $\aa^1,$ and marking the points $(x_1,\dots, x_n; 0_1, \dots, 0_f, 1_1,\dots, 1_f)$ given by images of marked points in $X$, and images of the points $0, 1$ in the various copies of $\aa^1$. This curve is canonically endowed with $f$ maps $\alpha_i:\aa^1\to \hourparam(X)$ which map $0,1\in \aa^1$ to the corresponding ``new'' marked points $0_i, 1_i\in \hourparam(X)$. Let $\Hour_{g,n,f}^{unst}$ be the the stack which classifies the data of $(Y\in \bar{\M}_{g,n+2f}; x_i, 0_j, 1_j; \alpha_j:\aa^1\to Y)$ with the obvious conditions, namely $\alpha_j(0) = 0_j$ and $\alpha_j(1) = 1_j.$ Then we have a map of stacks $$\hourparam: \bar{\ffc}_{g,n,f}^{unst}\to \Hour_{g,n,f}^{unst},$$ which is an isomorphism of stacks by an argument analogous to the one in Section \ref{sec:annuli}. Now it is obvious that  $\Hour_{g,n,f}^{unst}$ is an $\A_{0,1}^f$-torsor over $(\bar{\ffc}_{g,n,f}^{unst})^\wedge,$ and that this map preserves both stability (as every new glued component is stable) and strict instability (as strict instability is a property of the global irreducible component of a curve, which is not affected by any of our gluing procedures).
\end{proof}

It remains to understand $\bar{\ffc}_{0,0,1}$ of framed formal disks. Note that these will not be very interesting, as they cannot be nodal (indeed, none of the new results in this paper would change if we replaced $\bar{\ffc}_{0,0,1}$ with the one-point schem, classifying the formal disk with a standard framing). Nevertheless, we can give a model for $\ffc_{0,0,1}$ in terms of a variant of the hourglass construction. Indeed, let $\A_{0,1,-1}$ be the group of automorphisms of $\aa^1$ fixing $-1, 0, 1.$ Then we have an equivalence $\mathtt{half-hour}:\ffc_{0,0,1}\to \A_{-1,0,1}$ given by taking $D, \phi$ to the asymetrically glued space $D\sqcup_{\phi\sim\psi_\infty}(\aa^1),$ with markings at the images of $-1, 0, 1$. Then the resulting space will be $\pp^1$ with three marked points (which has no automorphisms or moduli), and conversely, the framing can be uniquely recovered from a map $\aa^1\to \pp^1$ which misses $\infty$ (equivalently, maps $\aa^1$ to $\aa^1$) and fixes $-1,0,1$. 

\section{Stabilization, stable gluing, annuli \dl{(a8)}}\label{sec:a8}
\subsection{Stabilization and stable gluing \dl{(a8.1)}}\label{sec:a8.1}
Using the last section, we can reformulate our stability condition as follows:
\begin{center}
$(X; x_i; \phi_j)\in \bar{\ffc}_{g,n,f}^{unst}(S)$ is stable iff $\hourparam(X)\in \bar{\M}_{g,n+2f}^{unst}$ is stable.
\end{center}
Now on the level of classifying closed curves, for $2g + n\ge 3$ there is a \emph{stabilization} map $\stabilize: \bar{\M}_{g, n}^{unst}\to \bar{\M}_{g, n}$ which is a left inverse to the open embedding $\bar{\M}_{g,n}\to \bar{\M}_{g,n}^{unst}$ (exhibiting a sort of non-separatedness property  of the stack $\bar{\M}_{g,n}^{unst}$). This stabilization map is compatible with gluing, and for any curve $X\in \bar{\M}_{g,n}^{unst}(S)$ there is a map $X\to \stabilize(X)$, compatible with marked points. Now the stabilization map extends in an obvious way to the hourglass space $\Hour_{g, n, f}$: indeed, if $\alpha_j:(\aa^1;0,1)\to (Y; 0_j, 1_j)$ is a global framing map, then it remains so after stabilization (as on the reduced locus, every component hit by a framing map is stable). This induces a stabilization map $\stabilize:\bar{\ffc}_{g,n,f}^{unst}\to \bar{\ffc}_{g,n,f}.$ We define new gluing maps $\glue_{jj'}: \bar{\ffc}_{g,n,f}\times \bar{\ffc}_{g',n',f'}\to \bar{\ffc}_{g+g', n+n', f+f'-2}$ and $\selfglue_{ij}:\bar{\ffc}_{g,n,f}\to \bar{\ffc}_{g+1, n, f-2},$ by composing the unstable gluing maps (which, recall, over any geometric point will blow out an unstable component of genus  zero) with the map $\stabilize.$ Importantly, the new stable gluing maps still satisfy the order-independence property we expect of gluing maps: namely, composing several gluing or self-gluing operations at once does not depend on the order of these operations (this is obvious for stable gluing from locality of the stabilization procedure over the stabilized curve). In particular, the stable gluing operation endows the stack of annuli $\bar{\ffc}_{0,0,2}$ with the structure of an associative semigroup in the category of stacks. We end this chapter with a brief ``lie-theoretic'' description of this semigroup: the semigroup of annuli behaves like a large-volume limit of global versions of the Witt algebra (equivalently, the Virasoro algebra at $c = 0$).

\subsection{A unital extension of $\bar{\ann}$ and a multiplication formula \dl{(a8.2)}}\label{sec:a8.2}
Note that the algebraic semigroup $\bar{\ann} (= \bar{\ffc}_{0,0,2})$ is non-unital, and it would be desirable for certain applications (such as operad structure on the $\bar{\ffc}_{*,*,*}$) to extend it to a unital monoid in a natural way. This will be the goal of the present chapter. Namely, we will define a \emph{unital} monoid $\bar{\A}_m$ which is unital and contains $\bar{\ffc}_{0,0,2}$ in a ``formally open'' sense (and further, in a sense to be seen later, does not change the motivic type). The monoid $\bar{\A}_m$ will act (in $f$ commuting ways corresponding to the framings) on each space $\bar{\ffc}_{g,n,f}$ in a way that extends the actions of $\bar{\ann}$ (by symmetric gluing). 

Note that the semigroup $\bar{\ann}$ has a complex-analytic counterpart defined in \cite{oancea_vaintrob}, the space $\t{NodAnn}$ of possibly nodal stable complex annuli: see \cite{oancea_vaintrob} for a definition. This semigroup is also non-unital, and admits a unital extension (called $\tilde{\t{NodAnn}}$ in \cite{oancea_vaintrob}), defined using ideas very similar to the construction of the present chapter. There is an interesting difference between these two extensions: while in the complex-analytic context, $\tilde{\t{NodAnn}}$ is a (partial) compactification of $\t{NodAnn},$ in the present algebro-geometric context the embedding of $\bar{\ann}$ inside $\bar{\A}_m$ can be described (ignoring monoid structure) as the embedding $\dd\times \dd^\infty\subset \aa^1\times \dd^\infty,$ i.e.\ the larger space extends one of the formal coordinates to a ``global one''.

\

As a pleasant consequence of the construction in the present chapter, we deduce an explicit formula for multiplication in $\bar{\ann}$ in coordinates, in terms of a certain ``large-volume'' limit of (a variant of) the Witt algebra $W : = \Gamma(T\gg_m)$ of vector fields on the punctured line. 

\

We construct $\bar{\A}_m$ as the partial compactification of a certain group stack (in fact, a group ind-scheme). 
\begin{defi}[\dl{a14}]\label{defi:a14} Define $$\A_m: = \aut^{conn}(\gg_m)$$ to be the identity connected component in the stack of automorphisms of the open scheme $\gg_m.$
\end{defi}
This is an infinite-dimensional formal Lie group with closed locus $\gg_m$ itself (acting by left multiplication) and Lie algebra given by algebraic vector fields on $\gg_m.$ The embedding of the closed locus gives rise to a short exact sequence in ind-group schemes $$1\to \gg_m\to \aut^{conn}(\gg_m)\to \aut^{conn}(\gg_m, 1)\to 1.$$ (If we were to avoid taking connected components, we would still get an exact sequence, but both $\aut(\gg_m)$ and $\aut(\gg_m, 1)$ would contain an extra ``antipodal'' connected component containing the automorphism $t\mapsto t^{-1}$.)
The last Lie group $\aut(\gg_m, 1)^{conn}$ is a formal Lie group with a single point, thus uniquely determined by its Lie algebra, which is the subalgebra of vector fields on $\gg_m$ spanned by $$\{x^n(x\frac{d}{dx})\mid n\neq 0\},$$ i.e.\ by all standard generators except for the equivariant one $x \frac{d}{dx}.$ 
\begin{thm}[\dl{thm:a7}]\label{thm:monoidext}\label{thm:a7}
There is a monoid $\bar{\A}_m$ in which the group $\A_m$ is an open submonoid, with reduced points $\aa^1$ (with multiplicative monoid structure), and such that there is a canonical isomorphism of monoids between $\ffc_{0,0,2}$ and the completion $\bar{\A}_m^\wedge$ in the neighborhood of the point $0\in \aa^1.$ Further, there are gluing maps $\glue_{j}:\bar{\A}_m\times \bar{\ffc}_{g,n,f}$ which extend, on the one hand, the map $\bar{\ffc}_{0,0,2}\times \bar{\ffc}_{g,n,f}$ of gluing an annulus to the $j$th marked component of a curve and, on the other hand, the map $\reparam_j:\A_g\times \bar{\ffc}_{g,n,f}\to \bar{\ffc}_{g,n,f}$ in which $\alpha\in \A_g$ acts by changing the framing $\phi_j$ by precomposing with the restriction of $\alpha$ to the punctured neighborhood of $0$. 
\end{thm}

\begin{proof}
Let $(\gg_m)_1, (\gg_m)_2$ be two tori with coordinates $t_1, t_2,$ respectively. Let $\tilde{\Hour}^\circ$ be connected component of the identity in the moduli stack of triples $$(U, \alpha_1:(\gg_m)_1\to U, \alpha_2:(\gg_m)_2\to U),$$ where $U$ is a space and $\alpha_1, \alpha_2$ are isomorphisms, which (on the reduced locus) take $0\in \bar{(\gg_m)}_1$ to $0\in \bar{U}$ and $\infty\in \bar{(\gg_m)}_2$ to $\infty\in \bar{U}$. Via identification of $U$ with $\gg_m$ via $\alpha_1$ and of $(\gg_m)_1$ with $(\gg_m)_1$ via $t_1\mapsto t_2$, this scheme is canonically isomorphic to the scheme $\A_m$ above. Now we extend this to the following space.
\begin{defi}[\dl{a15}]\label{defi:a15}
Let $\tilde{\Hour}$ be the space classifying triples $i_1:(\gg_m)_1\to U\from \gg_m:i_2$ with $U$ a stable nodal curve of genus zero with two punctures, with the requirement that $i_1$ takes the coframing $\psi_0$ to a coframing of $U$ and $i_2$  takes the coframing $\psi_\infty$ to a \emph{nonintersecting} (i.e., over any geometric point, bounding a different point of the glosure) coframing. We impose no condition on $\phi_1\circ \psi_\infty$ or  $\phi_2\circ \psi_0.$ 
\end{defi}

This space has monoid structure, with $(U; \alpha_1, \alpha_2)\circ (U'; \alpha_1', \alpha_2')$ given by stabilization of the space $U\cup_{\alpha_2\sim \alpha_1'}$ with two global components glued. To interpret its geometry, we cover $\tilde{\Hour}$ by two opens. One will be $\tilde{\Hour}^\circ,$ corrresponding to triples $(U; \alpha_1, \alpha_2)$ with no nodal fibers and the other given by the complement to the closed subset of $\tilde{\Hour}^{\circ}$ corresponding to pairs of parametrizations with $\alpha_1(1) = \alpha_2(1).$ Write $\tilde{\Hour}^{\neq}$ for this open. Then we have a map $\tilde{\Hour}^{\neq}\to \bar{\M}_{0,4}$ given by the glued space $$\bar{U}: = U\cup_{\alpha_1}(\aa^1)_{t_1}\cup_{\alpha_2}(\aa^1)_{t_2^{-1}},$$ and we mark the boundary points (images of $0$ in the two glued copies of $\aa^1$), along with the (by assumption distinct) points $\alpha_1(1), \alpha_2(1)$. Then $\tilde{\Hour}^{\neq}$ is canonically an $\A_{0,1}$-torsor over $\M_{0,4}\cup \{\del\}\subset \bar{\M_{0,4}},$ where $\del$ is the boundary component corresponding to the reducible nodal curve with $\alpha_1(\{0,1\})$ and $\alpha_2(\{1,\infty\})$ in different components. Its neighborhood of $\del$ is canonically identified with $\Hour_{0,0,2}$, in a way that is comptible with composition. Since the complement $\tilde{\Hour}^{\neq}$ is contained in $\tilde{\Hour}^\circ,$ we see that $\tilde{\Hour}$ is covered by these two smooth opens. Further, on the level of reduced points, the two opens give the cover of $\aa^1$ by the complement of $1$ (corresponding to $\tilde{\Hour}^{\neq}$ via an identification $\M_{0,4}$ with $\pp^1\setminus \{0,1,\infty\},$ with the boundary point $\del$ going to $0$) and the compatible identification of the reduced locus of $\tilde{\Hour}^\circ$ with $\A_m,$ whose reduced locus is canonically $\gg_m$ (given by the image of $1$ under automorphisms). 
\end{proof}
\begin{defi}[\dl{a16}]\label{defi:a16}
  Define $\bar{\A}_m: = \tilde{\Hour},$ viewed as a monoid. 
\end{defi}
We can express the gluing in another, more Lie theoretic way. Let $\A_m^+\subset \A_m$ be the formal subgroup given by exponentiating the positive Lie subalgebra of the (non-extended) Virasoro, spanned by $$x^n\cdot  x \frac{d}{dx},  \quad n\ge 1$$ and $\A_m^0\subset \A_m$ the algebra obtained by conjugation by the antipodal map (i.e. given by $x^{-n}\cdot x\frac{d}{dx},\quad n\le -1$). Let $\rho_{\pm}:\A_m^+\times \A_m^-\to \aut(\A_m)$ be the automorphism of the group scheme given by two-sided action. From monoidal structure on $\bar{\A}_m = \tilde{\Hour}$, this action extends to an action on $\tilde{\Hour}$, which, abusing notation, we also denote $\rho{\pm}$. Now observe that the action $\rho_{\pm}$ of $\A_m^+\times \A_m^-$ on $\tilde{\Hour}$ is free. Indeed, it suffices to see that the orbit of every reduced geometric point $p\in \aa^1\cong \tilde{\Hour}_{red}$ (we use the subscript $red$ intsead of $0$ to avoid confusion with $\tilde{\Hour}^0$). We consider three cases. First, if $p = 1$ is the identity of the semigroup (locally, a Lie group) then to first order, the Lie group $\a_m^+\times \a_m^-$ acts on the Lie group $\a_m$ by addition, and this action is free because $\a_m^\pm$ are disjoint as sub-vector spaces. Now after translating by $p$, freedom of action on any $p\in \gg_m$ is equivalent to freedom near $1$ of the conjugated action $\A_m^-\times p(\A_m^+)p^{-1}$; on the level of Lie algebras, this conjugation rescales the coordinate $x^n(d \frac{d}{dx})$ by $p^m\neq 0$, thus the action is still free. Finally, at the singular point $0\in \aa^1,$ the action corresponds to changing parameterizations of a singular annulus by a subgroup in $\aut(\ddo)\times  \aut(\ddo)$ orthogonal to $\aut(\dd)\times  \aut(\dd),$ hence is a free and transitive action on strictly nodal framed annuli. From this we deduce a coordinatization $\bar{\A}_m\cong \A_m^-\times \aa^1\times \A_m^+;$ the monoidal structure is then the unique algebraic extension of the multiplication of the group structure on $\A_m\cong \A_m^-\cdot \gg_m\cdot \A_m^+.$ Note the similarity  of this construction with normal-ordered product in vertex algebras. Considering the local neighborhood around $0$, this also leads to a nice Lie-theoretic interpretation of the semigroup structure on $\bar{\ffc}_{0,0,2}$.

This monoid acts in $f$ commuting ways on $\bar{\ffc}_{g,n,f}$: to each framing $\phi$, the group scheme $\tilde{\hourparam}^\circ\subset \aut(\gg_m)$ acts on the framing by reparametrization in a way compatible with the action by gluing of the non-unital semigroup scheme $\bar{\ffc}_{0,0,2}$: this compatibility (and the resulting unital action of the whole $\bar{\ffc}_{g,n,f}$) follows from the argument in the proof of Theorem \ref{thm:monoidext} above. 

Note that it is evident from the past two sections that one could eliminate working with formal curves in the definition of the $\bar{\ffc}_{g,n,f}$ (and of gluing between them), and rather consider \emph{global} curves with punctures and maps from $\gg_m$ understood as boundaries. We have chosen to work from the point of view of formal curves as we find it more intuitive and easier to relate to the complex-analytic picture. The formal group functor is also the one classified by the forgetful map $\bar{\ffc_{g,n+1,f}}\to {\ffc_{g,n,f}}$, as seen in the following section. 

\subsection{Universal families}
In addition to gluing maps, the $\ffc_{g,n,f}$ are related among themselves by two important ``universal family'' maps, which we will use to understand the geometry of more complicated moduli spaces in terms of simpler ones. These maps are ``families'' insofar as they are fibrations in a formal sense and their fibers solve certain \emph{relative moduli problems}. The first is the \emph{tautological family}, $\taut:\bar{\ffc}_{g,n,f}^{taut}\to \bar{\ffc}_{g,n,f}.$ This is the stack which associates to a test scheme $S$ the groupoid of pairs $X, s$, where $X\to S$ is a framed formal curve and $S\to X$ is a section. Its fiber over the geometric point classified by $(X, x_i, \phi_i)$ is the formal scheme $X$. The following theorem is analogous to the case of ordinary curves (either $\bar{\M}_{**}$ or $\M_{**}$).
\begin{thm}
  The classifying map $\bar{\ffc}_{g,n,f}^{taut}\to \bar{\ffc}_{g,n,f}$ is isomorphic to the map $\bar{\ffc}_{g,n+1, f}\to \bar{\ffc}_{g,n,f}$ given by forgetting the last point and stabilizing.
\end{thm}
\begin{proof}
  It follows from Theorem \ref{thm:a6} that the map $\bar{\ffc}_{g,n+1, f}\to \bar{\ffc}_{g,n,f}$ is given by basechange of the map $\bar{\ffc}_{g,n,f}\to \bar{\M}_{g,n+2f}^\wedge$ along the map $\bar{\M}_{g,(n+1)+2f}^\wedge\to \bar{\M}_{g,n+2f}^\wedge.$ The latter map is classified by the formal neighborhood inside the universal family $\bar{\M}_{g,n+1+2f}\to \bar{\M}_{g,n+2f}$ of the sub-family $\bar{\M}_{g,n+1+f}\to \bar{\M}_{g,n,f}$. The pullback of this neighborhood to the universal family over $\Hour$ is precisely the (formal) complement to the maps classified by $\aa^1\to \hour(S),$ in other words, the tautological family over $S$.
\end{proof}

The other important universal family is the \emph{universal rotor} (at the $i$th framing), classified by the \emph{rotor projection}, a map $\rotor_i:\bar{\ffc}_{g,n,f+1}\to \bar{\ffc}_{g,n+1,f}$ given by asymmetrically gluing a single copy of $\aa^1$ (with standard co-framing) and marking the image of $0\in \aa^1.$ This can also be interpreted in terms of the hourglass construction, via the following evident lemma. 
\begin{lm}\label{lm:firstrotor} Let $X\to S$ classified by $S\to \bar{\ffc}_{g,n+1,f}$ be a family of framed formal curves. Then a lifting of $S$ to $\bar{\ffc}_{g,n,f+1}$ is a pair consisting of 
\begin{enumerate}
\item A point, $y$ of $\bar{\M}_{g,n+(2+2f)}^\wedge$ over $\hour(S)\in \bar{\M}_{g,(n+1)+(2f)}^\wedge$;
\item A parametrization by $(\aa^1, 0, 1)$ of the new blown-out component containing $x_{n+1}, y$ (extending a parametrization of the complement of the node), and sending $0, 1$ to $x_{n+1}, y,$ respectively. 
\end{enumerate}
\end{lm}
In a single fiber of $\rotor_i,$ choices of $y$ are canonically parametrized by the formal neighborhood $\hat{x}_{n+1}$ of the last marked point and choices of parametrization are canonically a principal $\A_{0,1}$-bundle over $\hat{x}_{n+1}$. The curve parametrized by the fiber will have an extra node if and only if $y=x,$ i.e.\ the new blown out component is strictly nodal.

The rotor projection is a fibration (indeed, in a strong sense: fibers over any two $S$-valued point are \'etale locally isomorphic). Note that by chaining together several rotor projections we can understand the fiber of the map
$\rotor_{1,\dots, f} = \rotor_1\circ \dots\circ \rotor_f:\bar{\ffc}_{g,n,f}\to \bar{\ffc}_{g,n+f, 0}\cong \bar{\M}_{g,n+f}$ given by successively applying rotors. Indeed, the properties of $\rotor_i$ listed above imply that the fiber of the chained rotor map $\rotor_{1,\dots, f}$ over the curve $X\in \bar{\M}_{g,n+f}$ is an $\A_{0,1}^f$-torsors over the product of the formal neighborhoods of the last $f$ marked points. This map can also be understood via the hourglass construction. Namely, $\rotor_{1,\dots, f}$ is the composition of the map $\hour:\bar{\ffc}_{g,n,f}\to \bar{\M}_{g,n+2f}$ with the tautological map $\bar{\M}_{g,n+2f}\to \bar{\M}_{g,n+f}$ which forgets one of the two marked points associated to each framing. 

\section{Logarithmic schemes: basics}
The the theory of logarithmic schemes is the algebraic geometer's answer to the physicist's theory of fields with boundary conditions. In this section we will recall enough basic defitions and constructions having to do with logarithmic schemes for the purposes of defining and working with the log schemes $\bar{\ffc}_{g,n,f}$. We will relegate certain technical bits, having to do with normal-crossings log structures on formal schemes (and infinte-dimensional variants) to an appendix: see Appendix \ref{sec:formal_log}.

\subsection{Definitions}
Logarithmic schemes were introduced by Kato \cite{kato} as a generalization of the notion of a normal-crossings pair (and more general divisor pairs). A normal-crossings pair is a pair $(X,D)$ for $X$ a scheme and $D\subset X$ a normal-crossings divisor. The world of normal-crossings pairs, and more generally of logarithmic schemes, imitates and extends the world of varieties. In particular most of what one can do in geometry can be done for (sufficiently nice) logarithmic schemes and (arbitrary) normal-crossings pairs. There are notions of smooth and \'etale maps, pullbacks, forms and differentials, de Rham cohomology (see \cite{log_book}), and of coherent sheaves (see \cite{log_coh}). In many ways, the geometry on a normal-crossings pair $(X, D)$ models the geometry on the complement $X\setminus D$: for example, they have the same de Rham cohomology and \'etale site (in the sense of \cite{vologodsky}, they are motivically equivalent). In this paper logarithmic schemes allow us to make sense of the geometry of ``formal complements'' $X\setminus D$ where $D$ is a normal-crossings divisor in a \emph{formal} variety. Such complements do not make sense as classical geometric objects, as defined as the complement to a divisor in a formal variety might not be a formal variety itself (for example the formal disk $\dd\setminus 0$ is not a formal variety as we have seen before). On the other hand, for $D\subset X$ a normal-crossings formal subvariety, the normal-crossings variety $(X, D)$ does make sense, as does its underlying (``ind''-)logarithmic scheme. The geometry of this logarithmic object tells us what $X \setminus D$ \emph{should} be as a scheme. An additional advantage of logarithmic geometry in our context is that it will allow us to ``reduce'' the very large normal-crossings pair $\ffc_{g,n,f}: = (\bar{\ffc}_{g,n,f}, \bar{\ffc}_{g,n,f}^\nu),$ which is a pair of infinite-dimensional formal schemes, to obtain a logarithmic variety $\flc_{g,n,f}$ which is of finite type (with underlying scheme the reduced locus $\bar{\ffc}_{g,n,f}^0\cong \bar{\M}_{g, n+f}$). We will call this variety the moduli space of \emph{framed logarithmic curves} (of genus $g$ with $n$ marked points and $f$ framings). This space parallels in a remarkable way a construction of Kimura, Stasheff and Voronov, \cite{ksv}, which gives a finite-dimensional real model for the (infinite-dimensional complex-analytic) conformal operad of Segal, \cite{segal}. Specifically, to every log space $\X$ one can associate a \emph{Kato-Nakayama analytification} $\X^{an}$, which is a real analytic variety (in our case, a real manifold with boundary). The Kato-Nakayama space of the reduced moduli spaces of framed formal curves will coincide, in a way that respects gluing, with the KSV model. 

Our main reference for logarithmic geometry is \cite{log_book}. All log structures we will encounter are fine and saturated (or inductive limits of such). Say $X$ is a scheme. Then $\oo_X$ is a sheaf of rings on $X$ in both the Zariski and in the \'etale topology. We consider the latter. Let $(\oo, \cdot)$ be the sheaf $\oo_X$, but viewed as an \'etale sheaf of commutative \emph{monoids} under multiplication, ignoring addition. The subsheaf of invertible elements in $\oo_X,$ denoted $\oo^\times,$ is evidently a group scheme. The sheaf $\oo^\times$ is related to the sheaf of differentials, $\Omega,$ viewed as a sheaf of (additive) groups over $X$, by a beautiful map: namely, define the logarithmic differential $\t{dlog}:\oo^\times\to \Omega$ to be the map defined (locally) by $$\t{dlog}:f\mapsto \frac{df}{f}.$$ The foundational insight of log geometry is that if $D\subset X$ is a normal-crossings divisor then the map $\t{dlog}:\oo^\times \to \Omega$ extends to a larger sub-sheaf of the monoid $(\oo, \cdot)$, provided that one replaces the sheaf $\Omega$ of differentials by its twist, $\Omega(D)$, the sheaf of \emph{logarithmic differentials} with respect to $D$. Now the logarithmic differential extends to a map $$\t{dlog}_D: \oo^\times(D)\to \Omega(D),$$ where $\oo^\times(D)\subset (\oo, \times)$ is the monoid of functions which are invertible \emph{outside of $D$} (and may have arbitrary order of vanishing at $D$). Illusie and others observed that much of the geometry of the normal-crossings pair $(X, D)$ is captured by the monoid $\oo^\times(D)\subset (\oo, \cdot),$ the sheaf of logarithmic differentials $\oo(D)$ and the map $\t{dlog}_D$. In fact given any sub-monoid sheaf $\ll\subset (\oo, \cdot)$ with no locally zero sections, one can define a map $\t{dlog}_K: \ll\to \Omega\otimes K$ (for $K$ the sheaf of rational functions), then define $\Omega(\ll)$ to be the span over $\oo_X$ of the sheaf-theoretic image of this map. The initial idea of logarithmic geometry, then, is to generalize the normal-crossings picture and define a variety with ``logarithic boundary data'' as the data of a variety with a sub-monoid sheaf $\ll\subset (\oo, \cdot)$. It turns out that the theory behaves better (for example, admits well-behaved pullbacks) if instead of a sub-monoid, one allowed a more general monoid sheaf $\ll$ (in an \'etale sense) over $X$, together with a map of monoids $\ll\to (\oo, \cdot).$ 

This motivates the following definition, due to Kato.

\begin{defi} A prelog structure $(X, \ll, \alpha)$ on a scheme $X$ is a sheaf of monoids $\ll$ in the \'etale topology together with a map of monoids $\alpha:\ll\to (\oo_X, \cdot)$. 
\end{defi}
A logarithmic structure is a special kind of pre-log structure:
\begin{defi}
  A prelog structure $(X, \ll, \alpha)$ is a \emph{log structure} if the subsheaf $\alpha^{-1}(\oo^\times)\subset \ll$ is isomorphic to $\oo^\times$ via $\alpha.$
\end{defi}

Roughly, the purpose of a log structure is to isolate ``singular'' information carried by a pre-log structure, i.e.\ provide logarithmic behavior for functions in $(\oo_X, \cdot)$ which are not already invertible; if two pre-log structures differ only in their ``non-singular'' parts, they can be considered equivalent. In fact, a log structure can be considered an equivalence class of pre-log structures (similarly to how a sheaf can be viewed as an equivalence class of pre-sheaves), or in other words a pre-log structure on $X$ can be ``corrected'' to a fully logarithmic structure. The construction is as follows. 
\begin{defi}
  Given a prelog structure $\ll = (\ll, \alpha)$ on $X$, we define the \emph{associated log structure} $(\ll^+, \alpha^+)$ to be the sheaf $$\ll^+ : = \ll\times_{\alpha^{-1}\oo^\times} \oo^\times$$ (pushout in the category of abelian monoids), with map $\alpha^+$ induced from the maps of $\oo^\times, \ll$ to $\oo.$
\end{defi}
The functor $\ll\mapsto \ll^+$ is left adjoint to the forgetful functor $\forg:\mathtt{Log}\to \mathtt{PreLog}.$ As this functor is fully faithful, we have canonically $\ll^+\cong \ll$ if $\ll$ is already logarithmic and $(\ll^+)^+\cong \ll^+$ for pre-log $\ll$. 

A \emph{logarithmic scheme} is the data $\X = (X, \ll, \alpha)$ of a scheme with a log structure. When there is no ambiguity, we will suppress $\alpha$ and write $\X = (X, \ll)$.

Pre-log schemes form a category, in an evident way, and log schemes form a (full) sub-category, which we denote $\mathtt{LogSch}$. There is a pair of adjoint functors $$\mathtt{under}:\mathtt{LogSch}\leftrightarrows \mathtt{Sch}:\mathtt{triv}.$$ Here $\mathtt{under}$ is the forgetful functor that associates to the log scheme $\X: = (X, \ll)$ its underlying scheme $X$, and $\mathtt{triv}$ associates to a scheme $X$ the associated \emph{trivial log scheme}, with monoid $\ll = \oo^\times$ (mapping to $(\oo, \cdot)$ in the evident way). It follows from the adjunction that there is always a map $\pi:\X\to \mathtt{under}(\X)$ to the underlying scheme. 

The category of log schemes admits fibered products. An important special case is basechange along a map of schemes $X'\to X$ of a log structure $\X = (X, \ll)$ on $X$. It can be described alternatively as follows. Given a map of varieties $f:X'\to X$ and a log structure $\X = (\ll,\alpha)$ on $X$, define the \emph{pullback prelog structure} to be the pre-log structure associated to the monoid $\ll'_{pre} : = f^*(\ll)$ with map $\alpha'$ given by the composition $$[f^*\oo_X\to \oo_{X'}] \circ f^{-1}\alpha :\ll'_{pre}\to \oo_{X'}.$$ This is in general not a log structure. Define the \emph{pullback log structure} $(\ll', \alpha')$ to be the associted log structure $(\ll'_{pre},\alpha'_{pre})^+.$ The following easy proposition follows from universality of our constructions:
\begin{prop}
The log scheme $\X': = (X', \ll')$ defined above is canonically isomorphic to the pullback $\X\times_X X'$ in the category $\mathtt{LogSch}.$
\end{prop}
If the map $X'\to X$ is \'etale then the pullback pre-log structure is already a log structure (as the logarithmic property is \'etale local). We say that a map of log schemes $\X'\to \X$ is \emph{strict \'etale} (equivalently: $X'$ is an \'etale open in $X$) if it is isomorphic to the basechange of $\X$ along an \'etale map $U\to X$ to the underlying scheme. Later we will study a larger class of \emph{log \'etale} maps between logarithmic schemes. Strict \'etale maps over $\X$ form a topology isomorphic to the \'etale topology on $X : = \t{under}(\X)$. 

In our definitions and constructions so far we have taken the underlying space $X$ to be a scheme. However as we have been working \'etale locally, everything extends to the world of orbifolds in a straightforward way. In particular, we make the following definition. 
\begin{defi}
  Write $\mathtt{LogOrb}$ for the category of orbifolds $X$ with an \'etale sheaf of monoids $\ll$ and a map $\alpha:\ll\to (\oo_X, \cdot)$ which give a log structure when basechanged to any scheme $U$ along an \'etale map $U\to X$.
\end{defi}
The category of log orbifolds can be extended to a larger stacky context in several ways. We make the following (somewhat lazy, but adequate) definition. 
\begin{defi}
A log stack is a (quasi-)functor $\mathtt{LogSch}\to \mathtt{Gpd}$ to the category of groupoids which satisfies descent for the strict \'etale topology. 
\end{defi}

\subsection{Examples of log schemes}
Given a normal-crossings pair $(X,D)$ of varieties we define its associated logarithmic scheme as follows.
\begin{defi}
  The log variety $(X,D)^{log}$ associated to the normal-crossings pair $(X, D)$ has underlying scheme $X$ and sheaf of monoids $\ll_D: = \oo^\times(D)\subset (\oo, \cdot)$ consisting of functions which are invertible outside of $D$. 
\end{defi}
A normal-crossings \emph{orbifold} is a map of orbifolds $D\to X$ whose pullback to any \'etale open \emph{scheme} $X'$ is the embedding of a normal-crossings divisor in a smooth variety. The construction of the log scheme associated to a divisor is local, so associated to any normal-crossings orbifold we have a log orbifold. Such log orbifolds are (locally) \emph{smooth} (see \cite{log_book}, IV.1.1). 

The first and most important smooth log orbifold for us (in fact, a log orbifold associated to a normal-crossings pair) is the \emph{moduli space of log curves}, $\M_{g,n}^{log}$. Namely, let $\bar{\M}_{g,n}^\V\subset \bar{\M}_{g,n}$ be the normal-crossings sub-stack of the smooth stack $\bar{\M}_{g,n}$ consisting of strictly nodal curves. We make the following definition, due to Kato.
\begin{defi}[\cite{kato_curves}]
  Define the log orbifold $\M_{g,n}^{log}$ to be the orbifold associated to the normal-crossings pair $(\bar{\M}_{g,n}, \bar{\M}_{g,n}^\V).$
\end{defi}
Recall that for many purposes, a normal crossings logarithmic structure $(X, D)$ encodes information about the open sub-stack $X\setminus D$ --- for example, de Rham invariants of the two sides coincide. In this sense, $\M_{g,n}^{log}$ is a logarithmic stand-in for the open stack $\M_{g,n}$. In many contexts, $\M_{g,n}^{log}$ is a nicer space to work with than $\M_{g,n}$ as its invariants are ``more finite''. For example, there are Hodge-to-de Rham spectral sequence on both sides starting with E1 term the ``Hodge cohomology'' $\bigoplus_{p,q}H^p\Omega^q$ (interpreted in a logarithmic sense for $\M_{g,n}^{log}$) which converges to the same de Rham cohomology theory on both sides; yet, for $\M_{g,n}$ the terms in the E1 page are infinite-dimensional whereas for $\M_{g,n}^{log}$ the E1 page is finite-dimensional (and, at least in the $g=0$ case, degenerates at the E1 term).

Another important stack for us will be the \emph{log point}, $\pt_{log}$. Namely, let $\pt_{prelog}$ be the group scheme $(X, \ll, \alpha) = (\spec(k), \nn, 0)$ where the map of monoids $0:\nn\to (k,\cdot)$ takes every element to $0$ (and similarly for any \'etale extension of $\spec(k)$).
\begin{defi}
  We define $\pt_{log}: = \pt_{prelog}^+.$
\end{defi}
$\pt_{log}$ has a very useful interpretation via basechange. Namely, let $(\aa^1,0)^{log}$ be the log variety associated to the divisor $0\subset \aa^1.$ Let $i:\pt\to \aa^1$ be the embedding of the point at $0$. Then $\pt_{log}\cong \pt\times_{\aa^1} (\aa^1,0)^{log}$. 

There is a generalization of this construction to famlies over a base. Namely, let $\tau$ be a line bundle over a base $B$, with total space $\aa_\tau$ and zero section $\iota:B\hookrightarrow \aa_\tau$ isomorphic to $B$. 
\begin{defi}
The \emph{relative log point} $(B, \tau)^{log}$ associated to the bundle $\tau$ is the pullback $B\times_{\aa_\tau} (\aa_\tau, \iota(B))^{log}$. 
\end{defi}

An alternative point of view on relative log bundles: if $\tau/B$ is a line bundle, then equivariant functions on $\aa_\tau$ invertible outside the zero section define, locally over $B$, a sheaf of monoids $\ll_\tau$ which is an extension of $\nn$ by $\oo^\times.$ The relative log point is the log scheme $(B, \ll_\tau, \alpha)$ where $\alpha(f) = f\mid_B$ for locally constant $f$ and $\alpha(f) = 0$ for all functions of higher degree. 

It will be useful for us to observe that the normal-crossings pair $(\aa^1, 0)$ has unital semigroup structure (in the category of normal-crossings pairs), with identity the point $1\in \aa^1$ and multiplication given by the map $\aa^1\times \aa^1\to \aa^1$ (compatible with normal-crossings structure). This implies that $(\aa^1,0)^{log}$ is a semigroup object in the category of log schemes. Now $0\subset \aa^1$ is a \emph{non-unital} sub-semigroup, and the basechange $\pt_{log}$ inherits (non-unital) semigroup structure (another way of viewing this: $\pt_{log}$ is the ``universal'' log scheme over $\spec(k)$ associated to the semigroup $\nn,$ and its multiplication structure can be contravariantly deduced from the diagonal map $\nn\to \nn\times \nn$). This semigroup structure is closely related to the semigroup structure on $\bar{\ann}$. Note that a relative log point $(B, \tau)^{log}$ does not admit semigroup structure; instead, there is a product map $(B, \tau)^{log}\times (B, \tau)^{log}\to (B, \tau^{\otimes 2})^{log}.$

\subsection{Standard classes of log maps and schemes}
In this section we introduce and discuss the behavior of several niceness properties of log schemes and maps of log schemes (mostly by reference to \cite{log_book}). The one most relevant to this paper is the notion of \emph{log \'etale map} $\X'\to \X,$ a generalization of the strict \'etale property defined earlier. 

\noindent {\bf Fine and Saturater (fs).} See \cite{log_book} chapter IV.1.2 for a definition of fine and saturated (fs) log schemes. The property of being fine and saturated is a logarithmic analogue of the property of being normal and of finite type. 

\noindent {\bf Closed immersion.} A map $\X\to \Y$ is a closed immersion if it is (\'etale locally) isomorphic to a basechange of $\Y\to \mathtt{under}(\Y)$ by a closed immersion of the underlying scheme of $Y$. 

\noindent {\bf Log thickening.} A map $\iota:(X,\ll, \alpha)\to (X',\ll',\alpha')$ is a log thickening if 
\begin{itemize}
\item $\iota$ is a closed immersion
\item $X$ is reduced and the underlying map $X\to X'$ is (\'etale locally) a nilpotent thickening, given by the nilpotent ideal sheaf $I_{X'}.$
\item \'Etale locally on $X',$ the subgroup $1+I$ acts freely on $\ll'$. 
\end{itemize}

\

All log stacks we will encounter are (\'etale locally on the underlying scheme) fs, or ind-thickenings of such.

\

\noindent {\bf Flat map.} A map $X\to Y$ of log schemes is flat if (locally) it is so on underlying spaces. 

\noindent {\bf Smooth log scheme.} See \cite{log_book} Section ... for a definition of smooth maps. Over a point it can be reformulated as follows: a scheme over $\qq$ is log smooth if and only if it is of the form $(X, D)^{log}$ for $X\setminus D$ smooth and $D$ locally of \emph{toric type} (i.e.\ isomorphic to the embedding of the boundary inside a toric variety). Note that normal-crossings pairs are smooth but log points and relative log points are not (though they have smooth geometric realizations, as defined in the next section). 

\noindent {\bf Log \'etale map.} See \cite{log_book}, ... for definition of log tangent bundles. A map $$\X'\to \X$$ of (fine and saturated) log schemes is \emph{log \'etale} if it is finite flat on underlying schemes, and induces an isomorphism of log tangent bundles. (Note that the flatness condition is essential here, unlike for ordinary schemes. Weakening the flatness condition gives an important larger class of maps called log \emph{unramified} maps which we will not consider here.)

\noindent {\bf Relationships and invariance.} We have, as for ordinary schemes, the implications \'etale $\implies$ smooth $\implies$ flat. All three of these classes of maps are closed under fibered product and invariant under basechange. The property of being fs is invariant with respect to fiber product over the underlying scheme. All smooth schemes are fs. 

\subsection{The Kato-Nakayama analytification}
If $X$ is a scheme over the complex numbers $\cc,$ the complex variety $X^{an}$ is a topological space with points $X(\cc)$ with topology induced by the analytic topology on the complex numbers $\cc,$ and with analytic structure: a sheaf of rings over $\cc$ given locally by a suitable analytic completion of regular functions. Kato and Nakayama, \cite{kato-nakayama} extended this notion to log schemes in a way compatible with \'etale and de Rham invariants (see Section \ref{sec:mot-inv}). Namely, define $\cc_{log}$ to be the scheme $\spec{\cc}$ with log structure given by the monoid $\ll_\cc: = \rr^{\ge 0}\times S^1$, mapping to $(\cc, \cdot)$ as the norm and argument of a complex number. This is a surjective multiplicative map of monoids which can be seen geometrically as the real blow-up of the complex plane $\cc$ at $0$. Note in particular that it is one-to-one on $\cc^*$, hence defines a log structure on $\spec(\cc).$ If $\X$ is a log scheme, we define $\X^{an}$ as a set as $\X(\cc_{log})$ -- note that if the log structure on $\X$ is trivial this coincides with $\X(\cc)$. The compatible analytic topologies on $\cc$ and $\ll_\cc$ give $\X^{an}$ the structure of a topological space. For a coherent log structure on a scheme of finite type, this space is finite-dimensional and real analytic, but no longer complex-analytic (indeed, it is often real odd-dimensional). One can endow it with some ``mixed'' real and complex analytic structure (in nice cases, for example, it has a stratification whose strata are canonically real torus fibrations over analytic spaces), but we will not need anything more than the topological structure of $\X^{an}$ here. The (topological) \'etale site on $\X^{an}$ refines the log \'etale topology on the underlying scheme $X$, and so can be ``glued'' out of any \'etale cover of $X$. This allows us to define the \'etale site of the analytification $\X^{an}$ for $\X$ a complex log orbifold; this will evidently be a topological orbifold (a site glued out of an atlas of sites of the form $U^{\et}/G$ for $U$ a topological space and $G$ a finite group acting on $G$ -- see e.g.\ \cite{witten_orbifolds}). 

We will need the following standard fact about Kato-Nakayama analytic realizations of log varieties with strict normal-crossings log structure. Let $(X, D)$ be a normal-crossings pair. We abuse notation by identifying $X$ with $X^{an}$ (and similarly for $D$). Let $D_1, \dots, D_k$ be the normal-crossings strata. Let $L_i$ be the $S^1$-bundle associated to the normal bundle on each $D_i$ (the normal bundle $N$ defines a $\cc^*$-equivariant bundle $N^*$, and we define the $S^1$-bundle to be the quotient $N^*/\rr_+$). Now for our chosen ordering $D_1,\dots, D_n$ of divisors, consider the real blow-up $Bl_{D_1,\dots, D_k}^\rr(X)$ which we defeine to be the sequential blow-up of $X$ at $D_1,$ then the proper transform of $D_2$ (which, note, will be a real blow-up of $D_2$), then $D_3,$ and so on.
\begin{lm}
The real blow-up $Bl^\rr_{D_1,\dots, D_k}(X)$ is canonically isomorphic, over $X$, to $Bl^{\rr}_{D_{\sigma(1)}, \dots, D_{\sigma(k)}}(X)$ for any re-ordering $\sigma$ of the divisor components. 
\end{lm}
\begin{proof}
  Note that if an isomorphism over $X$ exists between two blow-ups as above then it is unique, as it extends the trivial automorphism of $U = X\setminus D$ which is (canonically identified with) a dense open in any blow-up as above. Thus canonicity is a non-issue, and further, if an isomorphism exists (topologically) locally over $X$ it automatically globalizes. Working analytically locally, we can cover $X$ by opens $U$ which only intersect $\le n$ distinct divisors (here $n$ is the dimension of $X$). The restriction $B_U$ of $B$ to $U$ is the sequential blow-up of $D_i\cap U.$ It only depends on the relative order of the divisors that intersect $U,$ so WLOG we can assume they are $D_1,\dots, D_k$ (for some fixed $k$). By possibly refining our cover, we can identify $(U, D_1\cup \dots\cup D_k)$ with $(\cc^n, H_1\cup \dots \cup H_k),$ for $H_i$ the $i$th complex hyperplane. Now we see inductively that the $j$-fold sequential blow-up of $(\cc^n, H_1\cup\dots\cup H_k)$ is $\ll_\cc^j\times \cc^{n-k},$ where recall $\ll_\cc$ is the topological monoid $\rr^+\times S^1$ (viewed as a topological space) mapping to $\cc$ as the blow-up of the origin. The proper transform of $D_{i}$ for $i> j$ is then identified with the subvariety where the $i$'th coordinate is set to $0$. Thus the full sequential blow-up is $\cc^n$ along the $D_i$ is isomorphic (over $\cc^n$) to the product $\ll_\cc^k\times \cc^{n-k}.$ Performing the blow-ups in another order evidently produces an isomorphic (over $\cc^n$) space, and by the canonicity observed above we are done. 
\end{proof}
We deduce that given any strict normal-crossings divisor $D\subset X$ in a scheme, there is a canonically defined space $$Bl_D^\rr(X)$$ given by the blow-up of $X$ along the irreducible components of $D$ in any order. This procedure is evidently \'etale local over $X$, thus extends to a definition of $Bl_D^\rr(X)$ for a non-strict normal-crossings divisor, as well as for a normal-crossings divisor on an orbifold. 
\begin{lm}
Suppose $(X, D)$ is a normal-crossings pair, and let $\X : = (X,D)^{log}.$ Then $\X^{an}$ is canonically diffeomorphic to the real blow-up $Bl_D^\rr(X)$. 
\end{lm}
  \begin{proof}
Algebraically \'etale locally in $(X),$ the pair $(X,D)$ can be identified with the restriction to an open of the pair $(\aa^n, D^{std})$ for $D^{std}$ the standard normal-crossings divisor. The log-scheme $(\aa^n, D^{std})^{log}$ is the $n$-fold product of $(\aa^1,0)^{log}$ with itself. Now it is a standard fact (see e.g. section 8.2 of \cite{log_geom_mod}) that $\big((\aa^1,0)^{log}\big)^{an}$ is the real blow-up of $\aa^1(\cc)$ at $0$. 
  \end{proof}

\section{Framed formal curves}

\subsection{Definition and properties of $\ffc_{g,n,f}$}
Recall that the stack $\bar{\ffc}_{g,n,f}$ has a closed sub-stack $D : = \bar{\ffc}_{g,n,f}^\V$ classifying families of strictly nodal framed formal curves. Furthermore, the stack $\bar{\ffc}_{g,n,f}$ has a finite \'etale cover by formal schemes in the sense of Appendix \ref{sec:app_formal_log}. Furthermore, in the language of Appendix \ref{sec:formal_log}, the sub-stack $D$ is (\'etale locally) a normal-crossings divisor, i.e.\ defined in the formal neighborhood of any point by the equation $x_1, \cdot, \dots, \cdot x_k = 0$ for a finite set of formal coordinates $x_1,\dots, x_k.$ 
\begin{defi}
  We define the logarithmic stack $$\ffc_{g,n,f} : = \big(\bar{\ffc}_{g,n,f},D\big)^{log},$$ the logarithmic stack associated (via Definition \ref{defi:nc_formal_log}) to the normal-crossings pair $(\bar{\ffc}_{g,n,f}, D)$.
\end{defi}

As they are evidently compatible with the normal-crossings structure, the maps $\glue:\bar{\ffc}_{g,n,f}\times \bar{\ffc}_{g',n',f'}\to \bar{\ffc}_{g+g',n+n', f+f'-2}$ and $\selfglue:\bar{\ffc}_{g,n,f}\to \bar{\ffc}_{g+1, n, f-2}$ extend to maps of log stacks, which we call by the same name:
$$\glue:\ffc_{g,n,f}\times \ffc_{g',n',f'}\to \ffc_{g+g', n+n', f+f'-2},$$
$$\selfglue:\ffc_{g,n,f}\to \ffc_{g+1, n, f-2}.$$

While the definition of the log stacks $\ffc_{g,n,f}$ in language of infinite-dimensional formal orbifolds may seem difficult to access, there is a much more direct way to understand the log stack structure in terms of a pullback of a normal-crossings logarithmic structure of finite type. 

Namely, it follows from Theorem \ref{thm:a6} that $\bar{\ffc}_{g,n,f}^\V$ is canonically the pullback of the normal-crossings boundary divisor along the map $\hour: \bar{\ffc}_{g,n,f}\to \bar{\M}_{g,n+2f}.$ The log structure associated to a normal-crossings divisor is compatible with pullback for formal orbifolds (Lemma \ref{lm:formal_log_fun}), giving us the following result.
\begin{lm}\label{lm:ffc-as-pullback}
  For a (finite-type) scheme $X$ with map $f:X\to \bar{\ffc}_{g,n,f}$, the pullback log structure $\ffc_{g,n,f}\times_{\bar{\ffc}_{g,n,f}}X$ is a log structure on $X$ isomorphic to the pullback of $$\M_{g,n+2f}^{log}\times_{\bar{\M}_{g,n+2f}}X,$$ where $\M_{g,n+2f}^{log}\to \bar{\M}_{g,n+2f}$ is the map from the log scheme $\M_{**}^{log}$ to its underlying scheme and $X\to \bar{\M}_{g,n+2f}$ is the composition $X\to \bar{\ffc}_{g,n,f}\to \bar{\M}_{g,n+2f}.$
\end{lm}
The map $\bar{\ffc}_{g,n,f}\to \bar{\M}_{g,n+2f}$ factors through the formal neighborhood $\bar{\M}_{g,n+2f}^\wedge$ of $\bar{\M}_{g,n}$ (embedded by gluing a copy of $\pp^1$ with three marked points onto $f$ of the marked points). Thus the map of formal schemes $\bar{\ffc}_{g,n,f}\to \M_{g,n+2f}^{log}$ factors through the log scheme $\M_{g,n+2f}^{log, \wedge}$ associated to the formal normal-crossings pair $\M_{g,n+2f}^\wedge, D$ (for $D$ the restriction to $\M_{g,n+2f}^\wedge$ of the boundary divisor on $\M_{g,n+2f}$). From Theorem \ref{thm:a6}, the map on underlying stacks is a torsor for the formal group $\A_{0,1}.$ As the map of log stacks is a pullback of this one, it has the same fibers and thus is also a torsor for the same group (now viewed as a group object in the category of log-stacks). In particular, it follows (though we will not use this fact) that $\ffc_{g,n+2f}$ is the inductive limit of ind-thickenings (in the sense of \cite{log_book}, section IV.2.1) of $\M_{g,n+2f}^\wedge$ (itself an inductive limit of thickenings of the space $\flc_{g,n,f}$ defined in the following section).

\subsection{The reduced schemes $\flc_{g,n,f}$}
Given a log scheme $\X$ on a non-reduced scheme $X$, we define its reduced log subscheme $\X^0$ to be the scheme $\X\times_X X^0$ (for $X^0$ the reduced subscheme of $X$). This procedure is functorial and compatible with pullback (consequently, also with product): it is the right adjoint to the embedding of the category of log schemes with reduced underlying space to the category of all log schemes. We extend this definition (via adjunction) to stacks: a stack $X$ is reduced if every map from a scheme $S$ to $X$ factors through a reduced scheme. 

\begin{defi}
Define $\flc_{g,n,f} : = \ffc_{g,n,f}^0$ to be the log reduced locus of the log scheme $\ffc_{g,n,f}$.
\end{defi}

\begin{prop}
    The log scheme $\flc_{g,n,f}$ has underlying space $\bar{\M}_{g,n+f}$ and log structure canonically given by the restriction of the moduli space $\M_{g, n+2f}^{log}$ of log curves to the closed subscheme $\bar{\M}_{g,n+f}\subset \bar{\M}_{g,n+2f}$ (with embedding, as before, given by gluing $\pp^1$ with three marked points at each of the last $f$ marked points). 
\end{prop}
\begin{proof}
  This is a consequence of Lemma \ref{lm:ffc-as-pullback}.
\end{proof}

Thus the reduced log schemes $\flc_{g,n,f}$ are easy to define simply in terms of moduli spaces. The advantage of defining them in terms of log structure on the very large spaces $\bar{\ffc}_{g,n,f}$ of framed formal curves is that they now immediately inherit gluing operations $\glue$ and $\selfglue$ and compatibilities between them.

\section{Analytification of $\flc_{g,n,f}$, universal families, and the KSV construction}\label{sec:univfam}

Our goal of this section is to exhibit some kind of geometric moduli problem solved by the analytifications $\flc_{g,n,f}^{an}$, and deduce a comparison to the Kimura-Stasheff-Voronov spaces and their gluing. 

\subsection{Geometric realization of $\M_{g,n}^{log}$}\label{sec:smoothblowup}
We begin by sketching a geometric picture of the geometric orbifold $(\M_{g,n}^{log})^{an}$, which is well-known. Namely, as $\M_{g,n}^{log}$ comes from the normal-crossings pair $(\bar{\M}_{g,n}, \bar{\M}_{g,n}^\V)$, its analytification is the real blow-up of $\bar{\M}_{g,n}^\V.$ The fiber of the map $(\M_{g,n}^{log})^{an}\to \bar{\M}_{g,n}^{an}$ (underlying this blow-up) over a point $m\in \bar{\M}_{g,n}^{an}$ (classifying the complex curve $X$) is diffeomorphic to the real torus $T: =(S^1)^k,$ where $k$ is the number of nodes in $X$. This identification is not canonical, and the analytification of the universal family $\taut:\M_{g,n+1}^{log}\to \M_{g,n}^{log}$ gives a nice canonical interpretation of $T$. Namely, let $\tilde{X}$ be the normalization of $X$, and let $\nu_1, \nu_1', \nu_2, \nu_2', \dots, \nu_k, \nu_k'$ be the pairs of points of $\tilde{X}$ in preimages of each node (note that there is no canonical order on the nodes or the preimages, so both our choice of indices and our choice of which preimage to call $\nu_i$ or $\nu_i'$ are arbitrary). Let $x_1,\dots, x_n\in \tilde{X}$ the (unique) preimages of the marked points. Let $Bl \tilde{X}$ be the blow-up of $\tilde{X}$ at all $\nu_i, \nu_i'$ and all $x_j$. The space $Bl\tilde{X}$ is a (complex) surface with $n+2k$ boundary components $B\nu_i, B\nu_i', Bx_j$, each canonically a torsor over $S^1$ (corresponding to complex action on tangent directions). Then $T$ can be understood as the torus whose points are collections, for each pair $\nu_i, \nu_i',$ of \emph{anti-equivariant} (with respect to the two $S^1$ actions) identifications of $B\nu_i$ and $B\nu_i'$. To each such set of identifications, $t\in T$, we associate the curve $X_{x, t}$ obtained from $\tilde{X}$ by gluing along the identifications defined by $t$ --- this is the fiber of the analytification of the tautological map $\taut^{an}: (\M_{g,n+1}^{log})^{an}\to (\M_{g,n}^{log})^{an}.$ Note that this fiber blows up every marked point and smoothly ``resolves'' every nodal point by blowing up its two normal preimages, then gluing. This provides a nice geometric intuition for the fact that the universal family of logarithmic curves over $\M_{g,n}^{log}$ given by $\taut$ is \emph{smooth} (see \cite{log_book}, section IV.1). 

\subsection{Universal families for $\flc_{*,*,*}$}\label{univ_fam_log}
Now $\M_{g,n}^{log}$ is canonically isomorphic to $\ffc_{g,n,0},$ which is reduced, hence also isomorphic to $\flc_{g,n,0}.$ The two universal families $\taut: \bar{\ffc}_{g,n+1,f}\to \bar{\ffc}_{g,n,f}$ and $\rotor: \bar{\ffc}_{g,n,f+1}\to \bar{\ffc}_{g,n+1, f}$ from Section \ref{sec:univfam} are evidently compatible with normal-crossings structure, thus define maps (denoted the same way) $$\taut:\ffc_{g,n+1,f}\to \ffc_{g,n,f}$$ $$\rotor: \ffc_{g,n,f+1}\to \ffc_{g,n+1,f}.$$ These have reduced analogues, $$\taut^0:\flc_{g,n+1, f}\to \flc_{g,n,f}$$ $$\rotor^0:\flc_{g,n,f+1}\to \flc_{g,n+1, f}.$$ The tautological maps $\taut:\M_{g,n+1}^{log}\to \M_{g,n}^{log}$ get taken under these isomorphisms to the maps $\flc_{g,n+1,0}\to \flc_{g,n,0}.$ Now (see Section \ref{sec:univfam}) we can view $\ffc_{g,n,f}$ as a family over $\ffc_{g,n+f, 0}\cong\M_{g,n+f}^{log}$ via the chained rotor map $\rotor_{1,\dots, f}: \ffc_{g,n,f}\to \ffc_{g,n+f,0},$ and similarly view $\flc_{g,n,f}$ as a family over $\flc_{g,n+f,0}\cong \M_{g,n+f}^{log}$ via the reduced version, $\rotor_{1,\dots, f}^0.$ Now fibers of the rotor map over a marked curve $(X, \{x_1,\dots, x_n\}, \{y_1,\dots, y_n\})$ (over a scheme $S$) are $\A_{0,1}^f$-torsors over the product of formal neighborhoods $\hat{y}_1\times \dots \times \hat{y}_f$. The normal-crossings structure on the fiber of $\bar{\ffc}_{g,n,f}$ over $X$ is then induced (by pulling back along the torsor) from the standard normal-crossings structure $(\hat{y}_i, y_i)$ on each formal neighborhood, multiplied by the boundary normal-crossings structure on $\bar{\M}_{g,n+f};$ an analogous statement holds true about the log structure on $\M_{g,n,f}$. This means that fibers of the reduced map $\rotor_{1,\dots, f}^0:\flc_{g,n,f}\to \M_{g,n+f}^{log}$ are torsors by the reduced group stack underlying $\A_{0,1}$, which is the trivial group stack, over the reduced log scheme underlying $\prod (\hat{y}_i, y_i)^{log},$ which is the relative log point associated to the line bundle $T_{y_i}$ on $\M_{g,n+f}$ given by the tangent line at $y_i.$ 

\subsection{Geometric moduli realization of $\flc_{g,n,f}^{an}$}\label{geom_realn}
As we have a good geometric understanding for the ``classification problem'' solved by the analytication $\flc_{g,n,0} \cong \M_{g,n}^{log},$ in order to get a good interpretation of $(\flc_{g,n,f}^{log})^{an}$ it remains to geometrically understand the fibrations $\rotor_{1,\dots, f}: \flc_{g,n,f}\to \flc_{g,n+f,0}$. We see from \ref{lm:firstrotor} that, up to inductive thickenings (which are ignored by $\flc_{*,*,*}$) the fiber of $\rotor_i$ is canonically isomorphic to the relative log point on $\M_{g,n+1,f-1}^{log}$ classified by the line bundle $Tx_i$ tangent to the $i$th marked point. In particular, we deduce the following result.
 \begin{thm}
Via the map $\flc_{g,n,f}^{log}\to \M_{g,n+f}^{log},$ the fiber of the space $\flc_{g,n,f}$ over an $S$-point, $X\in \M_{g,n+f}^{log}$ is canonically isomorphic to the iterated log point associated to the collection of line bundles $(Tx_1, \dots, Tx_f)$ given by tangent bundles of the sections $x_1,\dots, x_f$. 
\end{thm}

After taking geometric realizations, then, the topological space $\M_{g,n,f}$ classifies the following data. 
\begin{enumerate}
\item A nodal marked curve $X^\nu.$
\item A resolution of each node by real blow-ups, as in section \ref{sec:smoothblowup}. 
\item\label{egaw3} A basepoint of the real blow-up of $X$ at each of the $f$ marked points corresponding to a ``framing'' (equivalently, a parametrization of the boundary component by $S^1$ compatible with the standard $S^1$-torsor structure of a real blow-up of a smooth one-dimensional complex variety).
\end{enumerate}

Note that given a pair of curves $X, X'$ as above corresponding to points of the analytic realization of $\flc_{g,n,f}, \flc_{g',n', f'}$ (respectively, a single curve $X$) any two framings can be glued in a unique way that identifies the basepoints (picked out in part \ref{egaw3} above) and is anti-equivariant for the $S^1$ action, to produce a new curve with precisely the data necessary to identify a point in $\flc_{g'', n'', f''}$ where $(g'', n'', f'') = (g+g', n+n', f+f'-2)$ (respectively $(g+1, n, f-2)$). Call the resulting curve $\glue^{geom}(X, X')$ (resp., $\selfglue^{geom}(X)$). Note that these operations were defined on the level of moduli, in an evidently continuous fashion, therefore give maps of topological orbifolds. 

The following lemma should come as no surprise.
\begin{lm}
The analytic realizations of $\glue:\flc_{g,n,f}\times \flc_{g', n', f'}\to \flc_{g+g', n+n', f+f'-2}$, respectively, $\selfglue:\flc_{g,n,f}\to \flc_{g+1, n, f-2}$ area equivalent (as maps of topological orbifolds) to $\glue^{geom}, \selfglue^{geom},$ respectively (as defined above). 
\end{lm}
\begin{proof} 
By looking in a formal neighborhood, it is enough to prove this lemma for gluing two annuli. Now the semigroup stack $\bar{\ann}$ is the neighborhood as $0$ of the monoid stack $\bar{\A}_m$, which is an ind-thickening of the monoid scheme $\aa^1$ (under multiplication). Thus the gluing operation on $\flc_{0,0,2}$ is no other than the semigroup structure on the log point $\pt_{log}$, with geometric realization given by multiplication on the circle. The lemma follows. 
\end{proof}

Now the spaces $\flc_{g,n,f}^{an}$ we described, together with their gluing maps, have been defined in the literature before. These are \emph{Kimura-Stasheff-Voronov} (KSV) spaces defined ans studied in \cite{ksv}. In particular, in \cite{ksv} it is proven that these spaces together with their gluing maps are linked by a chain of homotopy equivalences to the spaces of \emph{framed conformal surfaces}: oriented conformal surfaces (equivalently, complex surfaces) with boundary, and with every boundary component parametrized by $S^1$ in an analytic way. (This can alternatively be proved by techniques of \cite{oancea_vaintrob}.)

Some corollaries: 
\begin{cor}
The geometric realization of the stack-valued operad $O$ with $O_n : = \sqcup \flc_{g,0,f+1}$ and composition given by gluing is linked by a chain of homotopy equivalences to the bordism modular operad $\mathrm{Bord}^{2,1}$. (See \ref{sec:app_mod_op} for details.)
\end{cor}
In particular, this concludes the proof of Theorem \ref{thm:log_bord} from the introduction. 
\begin{cor}\label{cor:log_fld}
The geometric realization of the (ordinary, resp., cyclic) operad formed by gluing maps between $\flc_{0,0,f}$ is linked by a chain of homotopy equivalences to the (ordinary, resp., cyclic) operad of framed little disks, $FrE_2^{cyc}$. 
\end{cor}
In particular, we have proven Theorem \ref{thm:log_fld} from the introduction, with $FLD^{log}$ the operad with spaces $FLD^{log}_n : = \ffc_{0,0,n+1}$ and operad structure given by glueing.
\begin{cor}
The geometric realization of the category fibered in stacks with objects indexed by positive numbers $\langle n\rangle$ and with $\hom(\langle n\rangle, \langle m\rangle)$ given by $\sqcup_g \flc_{g,0,m+n}$ (and composition given by gluing) is linked by a chain of homotopy equivalences to the Segal category, defined in \cite{segal} and responsible for ``topologogical conformal field theory''. 
\end{cor}
\begin{cor}
The modular operad formed by the $\flc_{g,0,f}$ under gluing (see Appendix \ref{sec:app_mod_op}) has geometric realization equivalent to the modular operad formed by complex surfaces with parametrized boundary under gluing.
\end{cor}

\section{Applications}
We have seen that Theorems \ref{thm:log_fld} and \ref{thm:log_bord} follow from identifying operad structures given by gluing on $\ffc_{*,0, *}$ with the corresponding topological operads (either $\mathrm{Bord}^{2,1}$ or $FLD$) via a chain of quasiisomorphisms. It remains to deduce a log motivic structure on the operad $LD$ of little disks, and to use the motivic structures to produce weight splitting and (in the genus zero cases of $LD$ and $FDL$), formality. 
\subsection{Little disks from framed little disks}
Let $G$ be a (unital) monoid. Let $\C_G$ be the category with a single object, $\Theta,$ with morphisms $G$, and let $\C_G^\sqcup$ be the free symmetric monoidal category (under the operation $\sqcup$) generated by $\C_G$. Its objects are $\Theta^{\sqcup n},$ and morphisms generated by $\hom(\Theta^{\sqcup n}, \Theta) : = G^n$. This category is canonically the symmetric monoidal category of an operad, which we call $\comm_G,$ with $(\comm_G)_n: = \hom_{\C_G^{\sqcup}}(\Theta^{\sqcup n}, \Theta)\cong G^n$ (permutation of factors then comes from symmetric monoidal structure, and composition is induced by composition in $\C_G^\sqcup$). This construction is a purely formal one, and generalizes for $G$ a monoid object in any symmetric monoidal category. If $G = *$ is the trivial group (in any symmetric monoidal category), then $\comm_G$ is the terminal operad $\comm$, with one operation in every degree ($E_\infty$ is a standard resolution of this operad in topological spaces). More generally, $\comm_G$ is the operad that encodes the structure of $G$-equivariant commutative algebra. Functoriality gives us a map $\comm\to \comm_G$, given by mapping to the $n\to 1$ operation in $\comm_G$ given by the tuple $(1, \dots, 1)\in G^n$. Now if $O$ is an operad with equivariance by a groupoid $G$, the semidirect product operad (see \cite{gs_formality}) has map to $\comm_G.$ Rather than rigorously define semidirect product of operads, we explain the relevant example, in topological operads. Let $G=S^1$ and let $FLD$ be the operad of framed little disks (which is the semidirect product of $LD$ by the evident $S^1$-equivariant structure). Then given a configuration of framed little disks in $FLD_n,$ one can forget the disks and remember only their orientations, as $n$ elements of $S^1$. The resulting map of spaces $\angle:FLD_n\to (S^1)^n$ then evidently combines to an ``orientation'' map of operads $$\rho:FLD\to \comm_{S^1}.$$ Note that the little disks operad $LD$ is the pullback of the diagram
\begin{equation}\label{diag:fld_ld}
  \begin{tikzcd}
    FLD \arrow[r]& \comm_{S^1}\\
    & \comm \arrow[u]
  \end{tikzcd}
\end{equation}
Now the map $FLD_n\to (\comm_{S^1})_n$ is a Serre fibration at each level and the functor from operads to graded spaces detects fibrancy (in the standard Berger-Moerdijk model structure, \cite{berger-moerdijk}). Thus the above diagram is a homotopy pullback diagram for the operad $LD$. We can therefore (see Appendix \ref{sec:mot-inv}) define a motive with Betti model canonically equivalent to $LD$ if we can construct a diagram of log operads with geometric realization equivalent to the above pullback diagram.

We realize this diagram in the category of log schemes as follows. The operad $\comm_{S^1}$ is represented by the log operad $\comm_{\ann^\sim}$, with $\ann^{sim}$ the unitally extended monoid of annuli (the log version of section \ref{sec:a8.2}). For each space $FLD_n^{log} = \ffc_{0,0,n+1}$ and index $1\le i\le n,$ define $\angle_i:\ffc_{0,0,n+1}\to \ann^\sim$ as follows. When $n = 1$ (and so $i = 1$), take $\angle_i: FLD_1^{log}\to \ann^\sim$ to be the identity map. When $n\ge 2,$ define $$\angle_i : = \forg_{1,\dots, \hat{i}, \dots, \hat{n+1}}\circ \rotor_{1,\dots, \hat{i}, \dots, \hat{n+1}},$$ the map represented by gluing in an \emph{unmarked} copy of $\aa^1$ (with standard coframing) into each framing except the $i$th and the ``outgoing'' $n+1$st, producing an annulus. The maps $\angle_i$ morally ``measure'' the angle between the incoming $i$th boundary and the outgoing $n+1$st. In particular they satisfy the following evident identity, for an operad composition (with $X\in FLD^{log}_m, Y\in FLD^{log}_n,$ over an arbitrary log base): $\angle_j(X\circ_i Y) = \angle_j(X)\cdot \angle_i(Y),$ where $i\le n, j\le m,$ and composition of angles is given by the monoid structure on $\ann^\sim.$ 

From this we deduce that the maps $$\angle: = (\angle_1, \dots, \angle_n):FLD^{log}_n\to (\ann^\sim)^n$$ combine to form a map of log operads $\angle:FLD^{log}\to \comm_{\ann}.$ Tracing through geometric models, we see this map is equivalent after taking geometric realizations to the map $FLD\to \comm_{S^1}.$ Thus we have proven the following result.
\begin{thm}\label{thm:fld_ld_lim}
The diagram of log operads 
\begin{equation}\label{diag:fld_ld_log}
  \begin{tikzcd}
    FLD^{log} \arrow[r, "\angle"]& \comm_{\ann^\sim}\\
    & \comm \arrow[u]
  \end{tikzcd}
\end{equation}
has homotopy pullback of geometric realizations equivalent to the operad $LD$ of little disks. 
\end{thm}
This provides our log motivic enrichment of the little disks operad. This completes the proof of Theorem \ref{thm:log_ld} from the introduction, using the Lemma \ref{lm:eq_log_lim} from Appendix \ref{sec:mot-inv}.

\subsection{Weight splitting and formality}
\'Etale (co)homology of a variety defined over $\qq$ is Galois equivariant on the level of cochains. Let $F \in \Gamma_\qq$ be a Frobenius element associated to a prime $p,$ and let $\ell\neq p$ be another prime. Then if $X$ is an algebraic variety, (co)chains valued in $\qqbar_\ell$ can be viewed as $\qqbar_\ell[F]$-modules with finite-dimensional cohomology. The derived category of such modules, viewed as a symmetric monoidal $\infty$-category, is canonically graded by generalized eigenvalue of Frobenius. If $X$ is an algebraic variety, then $C_*(X, \qqbar_\ell)$ has \emph{motivic} cohomology, and in particular all eigenvalues that appear nontrivially in the grading are Weyl algebraic integers, and this grading can be induces an $\nn$-grading of $C_*(X, \qq)$ by $F$-\emph{weight} (log based $\sqrt{p}$ of an absolute value), called the (derived) weight grading. Now the Galois action on the \'etale (co)homology of the $\flc_{g,n,f}$ also has a weight grading, by the followign lemma. Define $\M_{g,n,\vec{f}}$ be the moduli space of smooth curves of genus $g$ with $n+f$ marked points, and with a choice of nonzero tangent vector at the last $f$ markings. 
\begin{lm}\label{lm:is_algebraic} $\flc_{g,n,f}$ is related to $\M_{g,n,\vec{f}}$ by a chain of maps of fs log schemes over $\qq$ which induce equivalence on analytification. 
\end{lm}
\begin{proof} Let $\bar{\M}_{g,n,\vec{f}}$ be the moduli space of stable nodal curves of genus $g$ with $n+f$ marked points, and with a choice of (possibly zero) tangent vector at the last $f$ marked points. Let $D : = \bar{\M}_{g,n,\vec{f}}\setminus \M_{g,n,\vec{f}}$ be the complement, and let $\bar{\M}_{g,n+f}\subset \bar{\M}_{g,n,\vec{f}}$ be embedded as the mutual zero section of all tangent bundles. Then from the last section we have $$\flc_{g,n,f}\cong (\bar{\M}_{g,n,\vec{f}},D)^{log}\times_{\bar{\M}_{g,n,\vec{f}}}\bar{\M}_{g,n,f},$$ and so the arrows $$\flc_{g,n,f}\from (\bar{\M}_{g,n,\vec{f}},D)^{log}\to \M_{g,n,\vec{f}}$$ furnish the desired sequence of arrows.
\end{proof}
Note that the stack $\M_{0,\vec{n+1}}$ makes sense for $n = 1$ and is equivalent to $\gg_m$ as the data of a genus-zero curve with two marked points and a nonzero tangent vector has no automorphisms and or moduli; adding a second marked vector gives $\gg_m$ worth of additional freedom. In particular the above argument can be extended (and simplified) in teh case of $n=1$, where we get the chain of equivalences: $$\fld_{0,0,2}\cong \pt_{log}\to (\aa^1, 0)^{log}\from \gg_m.$$ 

Let $\conf_{\aa^1, \vec{n}}$ be the configuration space of $n$ points on $\aa^1$ with choices of tangent vector. Choosing a fixed tangent vector $\del_\infty$ to $\pp^1$ at $\infty$ we define the map $\conf_{\aa^1, \vec{n}}\to \M_{0,0,\vec{n+1}}$ by completing the configuration of $n$ tagent vector on $\aa^1$ by the tangent vector $\del$ at $\infty.$ Then the map $\conf_{\aa^1, \vec{n}}\to \M_{0,0,\vec{n+1}}$ is equivalent to the quotient map $\conf_{\aa^1, \vec{n}}\to \conf_{\aa^1, \vec{n}}/\gg_a$ with action given by translation. In particular, the $\gg_a$-equivariang function $\vec{\angle}: \conf_{\aa^1, \vec{n}}\to \gg_m$ taking a configuration to its $n$th tangent vector (translated to $0$) factors through a map
$$\vec{\angle}_j: \M_{0,\vec{n+1}}\to \gg_m$$. Extending the argument in the proof of the lemma, we get the following corollary. 
\begin{cor}\label{cor:is_algebraic}
The map $\angle_j:\flc_{0,0, f+1}\to \flc_{0,0,2}$ is equivalent, via a sequence of maps of pairs of log varieties which induce equivalences on analytification, to the map $\vec{\angle}_j: \M_{0,\vec{n+1}}\to \gg_m$.
\end{cor}

 This implies that the $\qq_\ell$-valued \'etale chains on $\flc$ (equivalently, on $\ffc$) are ``motivic'', i.e. are equivalent to chains on an ordinary Deligne-Mumford stack (in turn, after possibly extending coefficients to a finite extension of $\qq$, equivalent to a finite colimit of invariants of \'etale chains on schemes with respect to finite group actions). Moreover they are complements to normal crossings divisors in smooth Deligne-Mumford stacks defined over $\zz,$ and so these representations are in the derived category of Galois representations generated by cohomology of schemes with potentially good reduction, and the $F$-weight grading splits the weight filtration (in the Hodge theoretic sense, see Deligne's \cite{trois_points}). As the Galois action is compatible with gluing maps, we deduce a split weight filtration on the entire modular operad $B^{1,2},$ proving Theorem \ref{thm:log_bord_weight}. Now note that $H^{et}_*(FLC_n^{log}, \qq_\ell)$ is a tensor product of several copies of $H_*(\gg_m)$ with the \'etale homology of an iterated fibration with fibers punctured affine lines. A spectral sequence argument then tells us that $H^{et}_k(FLC_n^{log}, \qq_\ell)$ has pure Frobenius weight $2k$. As weights of Frobenius behave in a symmetric monoidal manner, we deduce a formality splitting of the operad $C_*(FLC,\qq_\ell)$, which implies formality of $C_*(FLC)$ over any characteristic zero field. 

A similar argument (using a chain of log motivic equivalences between $LD_n^{log}$ and the space $\t{Conf}_n$ of $n$-point configurations on the affine line) gives a formality splitting of $C_*(LD,\qq_\ell).$

\begin{appendices}
\section{Formal log structure}\label{sec:formal_log}\label{sec:app_formal_log}
Let $X$ be a formal scheme. 
\begin{defi}\label{defi:formal_log}
A log structure over $X$ is an \'etale sheaf $\ll$ of topological monoids together with an open map $\alpha:\ll\to (\oo_X,\cdot)$, such that for any geometric point $x\in X$ there exists in some formal \'etale neighborhood of $x$ a nilpotent ideal sheaf $\I$ with the property that the closed subscheme cut out by $I$ is of finite type and the sheaf of monoids $1+\I\subset \oo^\times$ acts freely on $\ll,$ with discrete quotient.
\end{defi}
For $\X = (X, \ll, \alpha)$ as above, say a closed subscheme of finite type $X'\subset X$ containing the reduced locus is \emph{thick} for $\ll$ if it satisfies the condition above. Evidently, a thickening of a thick scheme is thick. For $X'\subset X$ a thick scheme, define $\ll': = \ll/1+\I,$ with induced map $\alpha':\ll'\to \oo_{X'}$. Then for two thick schemes $i:X'\subset X''\subset X,$ we have canonically $i^*(X'', \ll'', \alpha'') \cong (X',\ll', \alpha')$ (using evident notation). Define $\X_{ind}$ to be the log stack defined as the direct limit of the $(X', \ll', \alpha')$ as $X'$ ranges over thick closed subschemes. It has underlying stack $X$ (viewed as an ind-scheme) and for any map $S\to X$ for $S$ of finite type, the pullback $S\times_X\X$ can be computed as $S\times_{X'}\X'$ for $X'$ any thick closed subscheme containing the image of $S$. 

For $D\subset X$ a closed subscheme, we say that $(X, D)$ is a formal normal-crossings pair if $X$ is formally smooth (possibly infinite-dimensional), and $D$ is \'etale locally cut out by a product $x_1\cdots x_k$ (some finite $k$), for $x_i$ formal local coordinates with linearly independent normals. We say that a map $(X,D)\to (X', D')$ of formal normal-crossings pairs is a map $X\to X'$ such that the pullback of $D$ is a finite thickening of $D'.$ 

\begin{defi}\label{defi:nc_formal_log} 
The formal log structure $(X,D)^{log, fml}$ associated to the formal normal-crossings pair $(X, D)$ is the subsheaf of $\oo$ consisting locally of functions of the form $\prod x_i^{n_i}\cdot f,$ where $x_i$ are local coordinates which cut out components of $D$ (as above) and $f\in \oo^\times.$ Define $$(X,D)^{log}: = (X,D)^{log, fml}_{ind}.$$
\end{defi}
As $\oo^\times$ acts freely on this monoid, any finite-type subset containing $X^0$ will be thick, hence $\ll$ as above is a formal log structure. 
\begin{lm}\label{lm:formal_log_fun}
The assignment $(X,D)\mapsto (X,D)^{log}$ is symmetric monoidal functorial in the pair $(X, D)$, and compatible with smooth pullback on the base $X$. 
\end{lm}
It suffices to prove that the assignment $(X,D)\mapsto (X, D)^{log, fml}$ is functorial. Note that for $f\in \oo,$ we have $f\in \ll$ if and only if the vanishing locus of $f$ is contained in a finite thickening of $D$, and this implies the lemma.

If $(X, D)$ is a normal-crossings pair in an \'etale local sense, we can reproduce the above construction.

\section{Log motivic structures on Betti (co)homology}\label{sec:mot-inv}
For $\Lambda$ a commutative ring, define $BR^{log}_\Lambda$ to be the \emph{Log Betti ring} functor $BR^{log}_\Lambda:\LogSch\to E_\infty\t{-Alg}_\Lambda$ as the functor $BR^{log}_\Lambda:\X\mapsto C^*(\X^{an},\Lambda)$ from log schemes to $E_\infty$ rings over $\Lambda.$ 
\begin{defi} We say that a \emph{Log Motivic realization} on Betti (co)chains with coefficients in a ring $\Lambda$ is an action $\alpha$ of a group $\Gamma$ on the DG functor $BR^{log}_\Lambda$. \end{defi}
Such an action can be interpted using an appropriate model-categorical language, or, better, in a higher-categorical context, as map of $(\infty,1)$-categories from $\Gamma$ to the full $\infty$-subcategory of the category $\t{Fun}(\LogSch, E_\infty\t{-Alg})$ of functors on the single object $F$. The latter definition extends to the context where $\Lambda$ is a derived ring and $\Gamma$ is an ``$\infty$-group'' (the loop space of a pointed $\infty$-groupoid). In particular, there is always a ``universal'' log motivic realization with coefficients in $\Lambda$ (given by taking the full group of automorphisms). 
\begin{defi}\label{log_mot_galois} We define the log motivic Galois group to be the prestack, i.e.\ the functor from rings to group (both in an $\infty$-categorical context) that takes a ring $\Lambda$ to the group of automorhpisms of $BR^{log}_\Lambda.$ \end{defi}
Note however the group in question may have non-discrete points and be hard to work with, so we work with more special and better-understood motivic realizations here. 

Our main example of a log motivic structure is a consequence of Kato and Nakayama's construction in \cite{kato-nakayama} of log \'etale cohomology. Here $\Lambda = \qq_\ell$ (or more generally $\zz_\ell$), and the group acting is $\Gamma = \Gamma_\qq,$ the absolute Galois group of $\qq.$ Note that as for any fs log scheme of finite type, its analytification is representable by a finite CW complex, the Betti cohomology functor has finite total cohomology and is dualizable as a $\zz$-module. In particular, motivic structures on Betti cochains compatible with product are equivalent to motivic structures on Betti chains compatible with coproduct. 

\subsection{Galois action on log Betti cochains with \'etale coefficients}
Define $\C_{\t{\'Et}}(\X)$ to be the category of ``log \'etale opens of $\X$,'' i.e.\ log \'etale maps $\X'\to \X$. A map $f:\X'\to \X$ (resp., a collection of maps $f_i\{\X'_i\to \X\}$) is a \emph{log \'etale cover} if it is \'etale and dominant on underlying schemes (resp., if the map $f:\sqcup \X'_i\to \X$ is a log \'etale cover). It is a straightforward check that log \'etale maps with log \'etale covers form a Grothendieck site. 

For a fs log scheme $\X$, the category $\C_{\t{\'Et}}(\X)$ endowed with this notion of cover is a site --- call this site $\t{\'Et}_\X.$ When $\X = X$ is a scheme, this site agrees with the ordinary \'etale site of $X$. Furthermore, log \'etale maps of fs log schemes imply \'etale maps of Kato-Nakayama analytifications, with covers going to covers, so we have a map of sites (in the opposite direction) $$\t{\'Et}_{an}:\t{\'Et}_{\X^{an}}\to \t{\'Et}_{\X}.$$

Kato and Nakayama \cite{kato-nakayama} showed that (over $\cc,$ and hence also over $\qqbar$) the \'etale homological invariants of this site agree with homological invariants of the Kato-Nakayama analytification; it follows from their proof that the isomorphism is given precisely by the map of sites above.
\begin{thm}[Kato-Nakayama, \cite{kato-nakayama}, Theorem 0.2(a)]
Suppose $X$ is an fs log scheme over $\cc.$ Then the map of sites $\t{\'Et}_{an}$ induces an isomorphism of cohomology groups $H^*_{\et}(\X, A)\cong H^*_{top}(\X^{an}, A),$ for any finite abelian group $A$.
\end{thm}
(In fact, Kato's proof extends to a result with non-constant ``constructible'' coefficient sheaf $A$.) It follows that the map of sites induces a chain-level isomorphism of derived global sections, and taking derived inverse limits along $\zz/p^n\zz,$ we deduce quasiisomorphism of (chain-level) \'etale cohomology with coefficients in $\zz_p$ and therefore $\qq_p$. 

We deduce the following result.
\begin{cor}\label{thm:log_betti_comparison}
  Let $\X$ be an fs log orbifold over $\cc$. Then we have (in the derived category) canonical quasiisomorphisms
  $$C^*_{\et}(X, \zz_p)\cong C^*(X^{an}, \zz_p)$$
  $$C^*_{\et}(X, \qq_p)\cong C^*(X^{an}, \qq_p).$$ 
\end{cor}
 \begin{proof} To pass from schemes to orbifolds we simply observe quasiisomorphism on each term in the \v{C}ech complex associated to a strict \'etale cover of $\X$ by log \emph{schemes}.\end{proof}

\subsection{Extending motivic realization to $LD$}
Note that our model for the operad $LD$ is not as an operad in log schemes, but as a ``homotopy limit'' of maps of such. More specifically, the spaces comprising the operad $LD$ are homotopy fibers of log schemes, i.e.\ objects of the category obtained from $\LogSch$ by formally adjoining homotopy limits. On the level of cochains, colimits of spaces go to direct limits of algebras; thus one can uniquely extend $LD$ to a functor from the category $\widehat{\LogSch}$ to $E_\infty$-algebras, in limit-preserving way.

\subsection{Proof of Theorem \ref{thm:log_ld}} 
This formalism us complete the proof of Theorem \ref{thm:log_ld} by using the following lemma.
\begin{lm}\label{lm:eq_log_lim}
Given a functor $B:\LogSch\to \C$ as above, the $\Lambda$-module \emph{co}-operad $C^*(LD,\Lambda)$ can be canonically lifted to a co-operad $FLD_\C$ in $\C,$ with a canonical equivalence on the level of spaces of operations $(FLD_\C)_n\cong B(\t{Conf}_n)$ for $\t{Conf}_n$ the space of configurations of $n$ distinct points in the plane (with trivial log structure). 
\end{lm}
\begin{proof}
Consider the diagram of operads in $\LogSch$ (from Theorem \ref{thm:fld_ld_lim})
\begin{equation}  \begin{tikzcd}
    FLD^{log} \arrow[r, "\angle"]& \comm_{\ann^\sim}\\
    & \comm. \arrow[u]
  \end{tikzcd}
\end{equation}
As homotopy pullback commutes with the functor $\O\to \O_n$ from operads to spaces, when symmetric monoidal structure is given by pullback, this can be understood as an operad in $\LogSch^{holim}.$ Moreover, by Corollary \ref{cor:is_algebraic}, on the level of spaces of $n$-ary operations, this diagram is equivalent to the diagram 
\begin{equation}  
\begin{tikzcd}
    \M_{0,\vec{n+1}} \arrow[r, "\vec{\angle}
"]& \gg_m^n\\
    & \pt. \arrow[u]
  \end{tikzcd}
\end{equation}
Now (in the terminology introduced for that corollary), the map $\vec{\angle}: = (\vec{\angle}_1,\dots, \vec{\angle}_n):\M_{0,\vec{n+1}} \to \gg_m^n$ is given by taking the $\gg_a$-quotient of the map $(\vec{\angle}_1,\dots, \vec{\angle}_n):\conf_{\vec{n+1}}\to \gg_m^n$ taking a configuration of $n$ points on $\aa^1$ with choice of nonzero tangent vector to the tangent vectors (where the tangent space is identified with $\aa^1$ also via translation equivariance). Now as a free quotient by $\aa^1$ induces homotopy equivalence on analytifications, the homotopy fiber of the diagram above is equivalent via a sequence of maps inducing equivalences on Betti cohomology to the homotopy fiber of $(\vec{\angle}_1,\dots, \vec{\angle}_n):\conf_{\vec{n+1}}\to \gg_m^n$, which is a fiber bundle (both in the category of schemes and after analytification), with fiber $\conf_n.$ 
\end{proof}

\subsection{Comparison with Hain's motivic Tate map}
In the paper \cite{hain}, Richard Hain constructs a map between two unstable motives corresponding to the \emph{Tate curve} parametrization (the parametrization of an elliptic curve by an infinitesimal genus zero surface), which is a map from the motive of $\gg_m\times \pp^1\setminus \{0,1,\infty\}$ to the motive of an infinitesimal neighborhood of the cusp in the universal punctured elliptic curve $\M_{1,2}$. Hain describes also a real geometric ``cartoon'' for this map: the domain is given by $S^1$ times the real blow-up of $\pp^1$ at $0,1,\infty$, and the codomain is the fiber over the cusp of the real blow-up of the identity section in the universal elliptic curve. The map is given by taking $(x, \theta)\in \t{Bl}_{0,1,\infty}(\pp^1)\times S^1$ to the curve obtained by glueing the boundary at $0$ and the boundary at $\infty$ along the angle $\theta,$ then marking the image of $x\in \pp^1\setminus \{0,1,\infty\}.$ Now we see (see \ref{geom_realn}) that this map geometrically is almost exactly the same as our glueing map $\selfglue :\flc_{0,2,2}^{an}\to \flc_{1,2,0}^{an}$ from a the moduli space of genus zero curves with two boundary components and two markings to the universal log elliptic curve. The minor difference is that instead of having realization equivalent to $\gg_m\times (\pp^1\setminus \{0,1,\infty\})$, our space $\flc_{0,2,2}$ has realization equivalent to $\gg_m^2\times \pp^1\setminus \{0,1,\infty\}$, as we are choosing boundary parametrizations for both of the boundary components we glue. This extra degree of freedom, however, can be eliminated after we pass to the homotopy limit completion. Indeed, $\flc_{0,2,2}$ is isomorphic as a log scheme to $(\pp^1, D)^{log}\times \pt_{log}^2,$ for $D = \{0,1,\infty\}\subset \pp^1$ the divisor. Let $\kappa: \flc_{0,2,2}\to (\aa^1,0)^{log}$ be the map given by the composition $\flc_{0,2,2}\mapsto \pt_{log}\mapsto (\aa^1, 0)^{log}$ given by first projecting to one of the log point coordinates, then embedding it in $\aa^1$ with log structure defined by the divisor at zero. Now we can define formally a homotopy fiber $F$ ``$\subset$'' $\flc_{0,2,2}$ of the derived limit-completion of $\LogSch$ given by taking the homotopy fiber of the diagram $\flc_{0,2,2}\rTo{\kappa} (\aa^1,0)^{log}\lTo{\iota}_1 \pt,$ where $\iota_1:\pt\to \aa^1$ is the embedding of the point at $1.$ On the level of analytification, taking a fiber of a map to $(\aa^1,0)^{log}$ is equivalent to fixing one of the two parametrizations, and by a corresponding algebro-geometric argument, $F$ is motivically equivalent to $\gg_m\times (\pp^1\setminus \{0,1,\infty\}).$ The restriction to $F$ of the map $\flc_{0,2,2}\to \flc_{1,2,0}$ is precisely the one described in the cartoon in \cite{hain}. 

Our logarithmic model does not furnish us with a noncommutative motive. However, the resulting map on Betti cochains $$(\selfglue|_F)^* :C^*(\flc_{1,2,0}^{an})\to C^*(F^{an})$$ is a map of rings with full log motivic structure. It is possible to apply a Tannakian duality procedure at a fiber $f\in F$ and the corresponding fiber $\selfglue(f)\in \flc_{1,2,0}^{an}$ (i.e. replace the rings by the groups of automorphisms of fiber functors, which may be derived objects), and from this to deduce a map on Lie algebras with full log motivic structure. We do not attempt to do this here, but perhaps if one assumes an identification between log and ordinary motives in this case and carefully performs a Tannakian duality argument, this can imply formally the result of the KZB computation in \cite{hain}, indeed in a way that guarantees compatibility of motivic structure with $\qq$ and even with $\zz$ coefficients.

\section{Integer coefficients}\label{sec:int_coeff}
For the sake of clarity, we have so far defined all our stacks to be over the rational numbers, $\qq.$ However with a little caution all the constructions go through over $\zz.$ In this section we sketch the relevant constructions extremely briefly.

It is known (see \cite{stacks_project}) that the moduli space of marked nodal curves is a smooth stack over $\zz.$ Our construction of the isomorphism $\hour:\bar{\ffc}_{g,n,f}\cong \Hour_{g,n,f}$ from Section \ref{sec:hourglass-gen} allows one to identify the stack $\ffc_{g,n,f}$ over $\zz$ with a bundle over $\bar{\M}_{g,n+2f}^\wedge$ with fibers $\A_{0,1}: = \aut(\aa^1, 0, 1)$. This is a formal group which is no longer determined by its Lie algebra. Nevertheless it is evidently true that $\A_{0,1}$ is the completion at the polynomial $x$ of the ind-affine monoid $M_{0,1}$ given by polynomials which fix $0,1,$ viewed as a monoid under composition. Thus once again $\bar{\ffc}$ is \'etale locally isomorphic to $\bar{\M}_{g, n+2f}^\wedge\times \dd^\infty$. The nodal locus will once again be a normal-crossings divisor (in an \'etale sense) --- see \cite{stacks_project}. Both our gluing construction and the constructions in Appendix B go through to give the modular operad in (suitably smooth) ind-stacks and modular operads in (ind) log stacks $\ffc, \flc$, with the former formally ind-smooth and the latter of finite type.

\section{Modular operads}\label{sec:app_mod_op}
Generalizations of the notion of operad are themselves often described as algebras over a suitable colored operad. We describe here a certain combinatorial construction of a colored operad, which we call $\mathtt{ModCor},$ the operad of ``modular corollas'', following Getzler and Kapranov, \cite{getzler-kapranov}. We explain briefly the notion of operad in the $\infty$-category world and of operad in (pre-)stacks. 

\subsection{The category of multicorollas}

In this section, the notion of \emph{graph} will denote a ``graph with half-edges'', a.k.a. a \emph{pointed} one-dimensional simplicial complex. This is a graph in the usual sense, with vertices indexed by a pointed set $V_+ : = V\sqcup\{*\}$ and (directed) edges $E$ with source, target functions $i, f:E\to V\sqcup \{*\}$, respectively. The vertex $*$ is interpreted as the ``empty'' vertex. For $\Gamma$ a graph, we define its geometric realization $|\Gamma|$ to be the "clopen" variety obtained by taking its realization as a one-dimensional simplicial complex, then removing the vertex $*$. Union of graphs is defined as the wedge union, relative $*$, which makes it compatible with union of geometric realizations. Connected components and genus of a graph are interpreted in terms of the associated geometric realization. Edges with (source, resp., target) $*$ are called (left., resp., right) \emph{half-edges}; an edge with both source and target $*$ is called an \emph{open edge}. We view ``classical'' graphs (without half-edges) as graphs in our sense by adjoining a disjoint base vertex $*$. In particular, when writing $\emptyset,$ resp., $\pt$ we will mean the graphs with vertices $*,$ resp., $\pt, *$ and no edges. A graph without half-edges or open edges is called \emph{closed}. The \emph{interior} of a graph with half-edges is the closed graph obtained by removing all open and half-edges. Write $E_+(\Gamma)$ for the set of ``outgoing half-edges'' $\{e_+\mid f(e) = *\}$ and $E_-(\Gamma)$ for the set of ``incoming half-edges'' $e_-\mid i(e) = *$. Then $E_+ \sqcup E_-$ has one element for each half-edge of $\Gamma$ with one non-empty vertex and two elements (the source and target ``halves'') for each open edge.

A graph is finite if it has finitely many edges and vertices. Finite graphs with graph isomorphisms form a category (indeed, a groupoid). We view it as a symmetric monoidal category with symmetric monoidal structure given by disjoint union. We also consider the larger category of \emph{undirected graphs}, $\Gr^\pm$ where isomorphisms between graphs are allowed to reverse directions of edges. We will be interested also in the category $\Gr_{mod}$ of ``modular graphs'', which for us will be synonymous with nonnegatively graded graphs: graphs whose interior vertices are decorated by weights in $\nn$. For reasons that will be aparent below, we call the weight associated to a vertex its \emph{genus}, and written $g(v)$. The underlying graph of a modular graph $(\Gamma, g:V_\Gamma\to \nn)$ is the graph $\Gamma$ with gradings forgotten. Isomorphisms of modular graphs are isomorphisms of underlying graphs which preserve genus of vertices. The category $\Gr_{mod}^\pm$ is defined similarly, with isomorphisms given by isomorphisms of underlying undirected graphs which preserve genus of vertices. 

A \emph{corolla} is a connected graph with interior isomorphic to $\pt$. It has a single interior vertex $v$ and edges which are (either left or right) half-edges abutting on $v$. A \emph{modular corolla} is a modular graph with underlying graph a corolla. A \emph{multicorolla} is a graph which is a disjoint union of corollas. A modular multicorolla is a modular graph with underlying graph a multicorolla. Note that isomorphisms of undirected graphs (i.e. edge re-orientations) preserve the property of being a (multi-)corolla. A \emph{labeled corolla} is a corolla with edges labeled $1,\dots, d$, with $d$ the degree (ignoring orientation). A labeled multicorolla is a multicorolla with vertices numbered $1,\dots, t$ (independent of genus) and each component corolla labeled as above.

Given a graph $\Gamma$ we define its corolla collapse $C(\Gamma)$ to be the multicorolla whose vertices are connected components $\pi_0(\Gamma)$, such that each $\alpha\in \pi_0(\Gamma)$ has incoming edges $E_-(\Gamma_\alpha)$ and outgoing edges $E_+(\Gamma_\alpha)$, for $\Gamma_\alpha\subset \Gamma$ the connected component classified by $\alpha.$ In other words: for every connected component $\alpha\in C(\Gamma)$ corresponding to a subgraph $\Gamma_\alpha\in \Gamma$ with at least one interior vertex $\alpha$ has edges in bijection with half-edges in $\Gamma_\alpha,$ and for every connected component $\Gamma_\alpha$ consisting of just an open edge $e$, the vertex $\alpha$ has one incoming and an outgoing half-edge corresponding to two ``halves'' of $e$. If $\Gamma$ is endowed with a grading $g:V(\Gamma)\to \nn,$ we define a grading on $C(\Gamma)$ with $$g(\alpha) : = \sum_{v\in V\big(\Gamma_\alpha\big)} g(v) + \t{gen}(\Gamma_\alpha).$$ Here $\t{gen}(\Gamma_\alpha)$ is the genus of the \emph{interior} of $\Gamma_\alpha$ (number of interior vertices minus number of interior edges). In particular, the grading of every $\alpha$ corresponding to an open edge is $0$. 

Given a graph $\Gamma$ and a set $S\subset E_{\t{int}}$ of edges \emph{possibly with multiplicity}, define the \emph{splitting} $\Gamma_{!S}$ to be the graph with homotopy type given by removing $n$ distinct points from the interior of an edge $e$ if it appears with multiplicity $n$ in $S$. Combinatorially, such an edge $e$ with $i(e) = x, f(e) = y$ (either or both of which can be $*$) gets split into an $n+1$-tuple of new edges, $e_0,\dots, e_n$ with $i(e_0)=x$, $f(e_n) = y$ and all other endpoints $*.$ For $\Gamma$ a graph, define the \emph{full splitting} $$\Gamma_!: = \Gamma_{!E(\Gamma)}$$ to be the splitting at all edges. This graph is a disjoint union $\Gamma_! = \Gamma_!^{in}\sqcup \Gamma_!^{out}$ where $\Gamma_!^{in}$ consists of a corolla for every vertex $v$ of $\Gamma$ (with the same multiplicity) and (canonically) $\Gamma_!^{out} \cong \Gamma_+\sqcup \Gamma_-,$ which is in bijection with edges of $C(\Gamma).$ 

Now we define the \emph{modular corolla category}, $\mathtt{ModCor}$ resp., the \emph{directed modular corolla category}, $\mathtt{ModCor}_{or}$, as follows. Objects are \emph{labeled modular multi-corollas}. A morphism from $M$ to $N$ is a graph $\Gamma,$ together with labelings of $\Gamma_!^{in}$ and $C(\Gamma)$ such that $\Gamma_!^{in}\cong M$ as an undirected (resp., directed) graph, and $C(\Gamma)\cong N$ as an undirected (resp., directed) graph. In both cases, we require the labellings to agree. Note that because of our labelings, neither objects nor morphisms have automorphisms. If $L\rTo{\Gamma'} M\rTo{\Gamma} N$ are morphisms as above, then the composition is the unique graph $\Gamma''$ (with appropriate labelings) such that there exists a collection $S$ of edges with multiplicity such that $\Gamma' = \Gamma_{!S}$ and $\Gamma = C(\Gamma_{!S}),$ compatibly with labelings. Informally, $\Gamma''$ is uniquely constructed by ``gluing in'' a component of $\Gamma$ at every vertex of $\Gamma'$ in a way compatible with vertices. Note that while compatibility of labelings insures the combinatorial objects classified by objects and morphisms have no automorphisms, nevertheless symmetries of corollas are encoded in the category. For example, if $\sigma$ is a re-numbering of the labels of a single corolla $M$, we can define a morphism with underlying graph $\Gamma = M$, with the labels of $\Gamma_!^{in}$ and $C(\Gamma)$ (both canonically isomorphic to $M$) differing by $\Gamma.$ A generalization of this argument includes the category of automorphisms of multicorollas as a sub-category of $\mathtt{ModCor}$ (and analogously for the directed variant).

\subsection{The colored operads of corollas.}
Define $$\t{Fin}$$ to be the category of \emph{labeled} finite sets, with objects $\langle n\rangle = \{1,\dots, n\}.$ We view $\t{Fin}$ as a symmetric monoidal category under disjoint union: $\langle m\rangle \sqcup \langle n\rangle : = \langle m+n\rangle,$ and this satisfies the requirements of symmetric monoidal structure in a standard way. It is good for intuition to think of objects of $\t{Fin}$ as abstract sets and to think of the operation $\sqcup$ as \emph{actual} disjoint union of sets (though this is not true and the symmetric monoidal structure involves shuffle maps, there is a way of weakening the notion of symmetric monoidal structure to make this intuition precise). 

It is clear that both $\mathtt{ModCor}_{dir}$ and $\mathtt{ModCor}$ are symmetric monoidal categories under \emph{disjoint union}, and that the assignment $C \mapsto \pi_0(C)$ taking a multicorolla to its set of connected components is a (strict) symmetric monoidal functor to the category $\t{Fin}$ of finite sets, also under disjoint union (note that the graph $\Gamma$ induces naturally a map of sets $\pi_0\Gamma_!^{in}\to \pi_0C(\Gamma)$). We call this functor $$\Pi: \mathtt{MultiCor}\to \mathtt{Fin}$$ (and similarly $\Pi:\mathtt{MultiCor}_{dir}\to \mathtt{Fin}$). 

Now suppose $\Pi: \C\to \mathtt{Fin}$ is a strictly symmetric monoidal functor (from some symmetric monoidal category). For $J\subset \mathtt{Fin}$ a subcategory, write $\C_J$ for the category of objects and arrows of $\C$ over objects and arrows of $J$, and in particular for $S\in \mathtt{Fin}$ a set write $\C_S$ for $\C_{*_S}$ with $*_S$ the category with one object and for $f:S\to S'\in \mathtt{Fin}$ a morphism, write $I_f\subset \mathtt{Fin}$ for the sub-category with two objects and one non-identity morphism $f$ and write $\C_f: = \C_{I_f}$ (so that, for $X\in \C(S), Y\in \C(S')$ we have $\C_f(X, Y) = \{\alpha:X\to Y\mid \Pi(\alpha) = f\}$).  

Now (see \cite{ha}, 2.1.0) the structure of a colored operad on a set $J$ of colors is equivalent to the structure of a pair $\C, \Pi:\C\to \mathtt{Fin}$ with certain properties, given in the following definition.
\begin{defi}
A pair $(\C, \Pi)$ with $(\C,\otimes)$ a small symmetric monoidal category and $\Pi:(\C,\otimes)\to (\mathtt{Fin},\sqcup)$ a symmetric monoidal functor is \emph{operadic} with color set $J$ if
\begin{itemize}
\item The category $\Pi_{\{1\}}$ has object set $J$. 
\item $\Pi$ is strict symmetric monoidal.
\item The functor $\Pi$ admits \emph{Cartesian lifts} over any isomorphism $\sigma:S\to S'$ of sets.\footnote{This is a requirement that ensures that the natural functors $\C_S\to \C_\sigma \from \C_{S'}$ are equivalences of categories, in a way compatible with mapping in or out, and corresponds to the symmetric group actions on the colored operad.}
\item For two disjoint sets $S, T$ the functor $\C_S\times \C_T\to \C_{S\sqcup T}$ (given by symmetric monoidal structure) is an equivalence of categories.
\item For $f:S\to S', g:T\to T'$, and objects $X\in \C_S, X'\in \C_{S'}, Y\in \C_T, Y'\in \C_{T'}$, the map induced by symmetric monoidal structure 
$\hom_f(X, X')\times \hom_g(Y, Y')\to \hom_{f\sqcup g}(X\sqcup Y, X'\sqcup Y')$ is a bijection.
\end{itemize}
\end{defi}

Here $\C, \Pi$ should be taken up to categorical equivalence. Given $\C,\Pi$ the underlying colored operad has operations $\t{Op}_\C(j_1,\dots, j_n\mid j) : = \hom_\C(j_1\otimes \dots \otimes j_n, j)$. Here the colors $j_i,j\in J$ are viewed as objects of the subcategory $\C_{1}\subset \C$, so $j_1\otimes \dots \otimes j_n\in \C_{\langle n\rangle}.$ Of course, any $\alpha\in \hom_\C(j_1\otimes \dots \otimes j_n, j)$ has $\Pi(\alpha)$ the unique map $\langle n\rangle \to \langle 1\rangle.$ We will use the terms \emph{operadic pair} and \emph{colored operad} interchangeably. 

This definition packs a lot of abstraction, but is good for generalizing to higher-categorical contexts. The following is a direct definition check.
\begin{lm}
The symmetric monoidal category $(\mathtt{MultiCor}, \sqcup)$ equipped with the functor $\Pi:\mathtt{MultiCor}\to \mathtt{Fin}$ is operadic, and similarly for $\mathtt{MultiCor}_{dir}$.
\end{lm}

Colors of the category $\mathtt{MultiCor}$ are corollas $C_{n\mid g}$ with $n$ inputs/outputs and genus $g.$ Colors of the category $\mathtt{MultiCor}_{dir}$ are $C_{m,n\mid g}$ with $m$ inputs, $n$ outputs and genus $g$. 

Before moving on, we define some (symmetric monoidal and operadic) sub-categories of $\mathtt{MultiCor}$ and $\mathtt{MultiCor}_{dir}.$ Note that given a colored operad $\O$ (equivalently, an operadic pair $\C, \Pi$) together with a subset of colors $J'\subset J$ (equivalently, $J'\subset \C_{\langle 1\rangle}$), there is a full colored sub-operad $\O_{J'}$ on these colors (equivalently, the full subcategory $\C_{J'}\subset\C$ tensor generted by $J'$ is once again operadic). We isolate several special cases that will be of interest to us. Define $\mathtt{MultiCor}_s\subset \mathtt{MultiCor}$ to be the span of ``stable corollas'', to be the full sub-category $\sqcup$ generated by all corollas with either $n\ge 1$ or $g\ge 2$ (i.e., excluding $n=0, g=0,1$), and $\mathtt{MultiCor}_{dir,s}$ to be the full sub-category $\sqcup$-spanned by all directed multi-corollas with $2(m+n)+3g\ge 3$. We say an oriented corolla $C_{g\mid m,n}$ is an oriented \emph{bush} if $g = 0$ and there is exactly one outgoing half-edge, i.e.\ $n = 1$. Respectively, an (unoriented) corolla $C_{g\mid n}$ is a bush if $g = 0$ and $n\ge 1$. Define $\mathtt{MultiBush}\subset \mathtt{MultiCor},$ resp., $\mathtt{MultiBush}_{dir}\subset \mathtt{MultiCor}_{dir}$ for the full subcategories spanned by bushes. Note that operations in $\mathtt{MultiBush}_{dir}$ are indexed by root-oriented trees with all vertices of genus zero (a pointed graph is a root-oriented \emph{tree} if there is a single outgoing half-edge, called the root, and a unique oriented path from any vertex or half-edge to the root). Operations in $\mathtt{MultiBush}$ are indexed by unoriented trees. 

Given an operadic pair $\O = (\C, \Pi)$ an \emph{algebra} over $\O$ in a symmetric monoidal category $(\V,\otimes)$ is a strict symmetric monoidal functor $\C\to \V.$ 
\begin{lm}
\begin{enumerate}
\item\label{xcv1} Algebras over $\mathtt{MultiCor}$ are modular operads.
\item\label{xcv3} Algebras over $\mathtt{MultiBush}_{dir}$ are operads.
\end{enumerate}
\end{lm}
\begin{proof}
For (\ref{xcv1}) see \cite{getzler-kapranov} and for (\ref{xcv3}) see \cite{heuts_hinich}. We will call algebras over $\mathtt{MultiCor}_s$ \emph{stable modular operads}. It would perhaps make sense to call algebras over $\mathtt{MultiCor}_{dir}$ \emph{directed modular operads}.
\end{proof}

The operads we defined map to each other as in the following diagram. 
\begin{equation}
\begin{tikzcd}
\mathtt{MultiCor}                 & \mathtt{MultiCor}_s \arrow[l]                   & \mathtt{MultiBush} \arrow[l]                  \\
\mathtt{MultiCor}_{dir} \arrow[u] & {\mathtt{MultiCor}_{dir,s}} \arrow[l] \arrow[u] & \mathtt{MultiBush}_{dir} \arrow[l] \arrow[u]
\end{tikzcd}
\end{equation}
Here the vertical maps are given by forgetting orientation. From this diagram of symmetric monoidal categories, we obtain a diagram of corresponding algebra categories, with arrows reversed. In particular, if $\O$ is a stable modular operad then we get ``for free'' a cyclic operad and an ordinary operad by pulling back the functor to the appropriate place in this diagram. 

\subsection{Modular operad structure on $\bar{\ffc}_{*,0,*}$}
Now we would like to define a stable modular operad structure $\O^{mod}_{\bar{\ffc}}$ on $\bar{\ffc}_{*,0,*}$, via the operations $\glue$ and $\selfglue$. While these operations (together with the identity) do generate the category $\t{MultiCor}_s$, we have two mild problems. The first is that $\bar{\ffc}_{0,0,2}$ does not have an identity element, but this can be fixed (as we have seen in section \ref{sec:a8.2}) by extending $\bar{\flc}_{0,0,2}$ to the monoid $\bar{\A}_m$ (in a way that, as we have seen, is compatible with gluing). The second is that the spaces $\bar{\ffc}_{*,0,*}$ are defined as stacks, i.e.\ most naturally objects of a two-category (and not an ordinary category). This is not a salient issue, or an interesting one: it is obvious that the gluing is canonical up to unique isomorphism, and can be upgraded to a real operad, but requires a little bit of extra formalism to formulate correctly. To do this we use $\infty$-operads, and we will use this language (as introduced, e.g., in \cite{htt}) for the remainder of this section. Define $\mathtt{PS}$ to be the category of $\infty$-prestacks, functors from the category of schemes to the $\infty$-category $Sp$ of homotopy types. 

Write $\mathtt{Fin}_+$ for the category of finite sets with partially defined maps (equivalently, the category of pointed sets). The corresponding $\infty$-categorical object is $N\mathtt{Fin}_+,$ the nerve of this category. Write $I_+$ for the set $I$ viewed as an object of $\mathtt{Fin}_+$. The category $\mathtt{PS}_\infty$ has direct products, and is symmetric monoidal under direct product. In the language of $\infty$-category theory, the symmetric monoidal structure is encoded as follows (see \cite{ha}, ...). Let $\mathtt{PS}_\infty^\oplus/\t{Fin}_+$ be the coCartesian fibration of categories corresponding to the functor $\langle n\rangle_+\mapsto \mathtt{PS}_\infty^{\times n}.$ Here to a partially defined map $S_+\to T_+$ given by $f:S'\to T$ for $S'\subset S$, define a functor $F_{f,S'}:\mathtt{PS}_\infty^S\to \mathtt{PS}_\infty^T$, with $F_{f,S'}(\{X_s\})$ having in its $t$th coordinate the object $\prod_{s\in f^{-1}(t)}X_s.$ Write $\mathtt{PS}^\otimes/N\t{Fin}_+$ for the Grothendieck construction associated to $F$. Then the functor $\mathtt{PS}^\otimes/N\t{Fin}_+$ is coCartesian and an $\infty$-operad, hence defines a symmetric monoidal structure.

Let $\mathtt{MultiCor}^+_s$ be the ($1$-)category whose objects are stable multicorollas $M$ and with $\t{Hom}(M, N)$ given by pairs $(M', \Gamma)$ for $M'\subset M$ a union of connected components of $M$ and $\Gamma\in \mathtt{MultiCor}(M', N)$ (with evident composition). The category $\mathtt{MultiCor}^+_s$ is almost the same as $\mathtt{MultiCor}_s,$ except morphisms are allowed to ``forget'' some connected components. The functor $\Pi$ extends to a functor $\Pi^+:\mathtt{MultiCor}^+_s\to \mathtt{Fin}_+$. By a simple definition check, the functor $\Pi^+$ (viewed as a functor of $\infty$-categories) is (after taking nerves) an $\infty$-operad: precisely the derived version of the one that defines the colored operad structure on corollas. 

Define a new category $\big(\O^{mod}_{\bar{ffc}}\big)^{\oplus}$ whose objects are triples $(C, S, X)$, where $C$ is a (labeled) multi-corolla, $S$ is a scheme and $X$ is a \emph{disconnected} framed formal curve over $S$ with components labeled by vertices of $C$, with appropriage genus, and framings of each connected component labeled by half-edges abutting the corresponding vertex (and no markings). Maps $(C,S,X)\to (C',S',X')$ are indexed by triples $(\Gamma, f, \iota= \sqcup\iota_*),$ where $\Gamma\in \mathtt{MultiCor}_s^+(C, C')$ is a graph, $S'\to S$ is a map of schemes and $\iota_*:X'\to f^*(X_{glue})$ is an isomorphism of (disconnected) framed formal curves, from $X'$ to the pullback to $S'$ of the curve $X_{glue}$ obtained by symmetrically gluing $X$ along the framings corresponding to each edge of $\Gamma$, but viewed as a collection of isomorphisms indexed by connected components of $C'$. It is a straightforward check that this category is coCartesian fibered in groupoids over the category $\mathtt{MultiCor}_s^+\times \mathtt{Schemes}^{op},$ hence (its nerve) defines a functor $N\mathtt{MultiCor}_s\times N\mathtt{Schemes}^{op}\to Sp$, equivalently $\mathtt{MultiCor}\to \mathtt{PS}_\infty.$ The decomposition $\iota = \sqcup \iota_*$ upgrades this to a functor $N\mathtt{MultiCor}\to \mathtt{PS}_\infty^\otimes,$ fibered over $\t{Fin}_+$, and it is a definition check to see that this is defines an algebra (in $\mathtt{PS}$) over the $\infty$-category $N\mathtt{MultiCor}_s/\t{Fin}_+$. 

Compatibility with normal-crossings structure and symmetric monoidicity of the functor from normal-crossings schemes to log schemes lets us deduce algebras $O_\ffc^{mod}, O_\flc^{mod}$ fibered in this category over $N\mathtt{MultiCor}_s$ in the $\infty$-category of pre-log stacks (i.e. the category of functors from $N\mathtt{LogSch}\to Sp$), associated with the corresponding logarithmic objects under gluing. 

In genus zero, the two-categorical issues vanish, and the induced cyclic operad, resp., operad $\O^{cyc}_{\bar{\ffc}}$ (as well as log analogues) are defined strictly, on the level of (ind-)schemes (in a way compatible with the $\infty$-categorical definition above). 
\end{appendices}

\end{document}